\documentclass[final]{siamltex1213}

\usepackage{ifpdf}
\ifpdf
  \usepackage{hyperref} 
\else
\fi

\usepackage{amssymb,latexsym,amsmath,verbatim,enumerate}
\newtheorem{remark}{{\em Remark}}
\newtheorem{example}{{\em Example}}

\newcommand{\eqdist}{=_{d}}
\def\indist{\to_{d}}
\def\E{\mathbb{E}}
\def\P{\mathbb{P}}
\def\<{\langle}
\def\>{\rangle}

 \title{Stochastic switching in infinite dimensions with applications to random parabolic PDEs}
\date{\today}
\author{Sean D. Lawley\footnotemark[2]\ \footnotemark[3]\ , Jonathan C. Mattingly\footnotemark[2]\ \footnotemark[4]\ , and Michael C. Reed\footnotemark[2]\ \footnotemark[5]}

\begin{document}
 
\maketitle

\renewcommand{\thefootnote}{\fnsymbol{footnote}}
\footnotetext[2]{Mathematics Department, Duke University, Durham, NC 27708 (lawley@math.duke.edu, jonm@math.duke.edu, and reed@math.duke.edu).}
\footnotetext[3]{This author was supported in part by NSF grant DMS-0943760.}
\footnotetext[4]{This author was supported in part by NSF grant DMS-0854879.}
\footnotetext[5]{This author was supported in part by NSF grants EF-1038593 and DMS-0943760 and NIH grant R01 ES019876.}

\renewcommand{\thefootnote}{\arabic{footnote}}

\begin{abstract}
  We consider parabolic PDEs with randomly switching boundary  conditions. In order to analyze these random PDEs, we consider more  general stochastic hybrid systems and prove convergence to, and  properties of, a stationary distribution. Applying these general  results to the heat equation with randomly switching boundary  conditions, we find explicit formulae for various statistics of the  solution and obtain almost sure results about its regularity and  structure. These results are of particular interest for biological applications as well as for their  significant departure from behavior seen in PDEs forced by disparate  Gaussian noise. Our general results also have applications to other  types of stochastic hybrid systems, such as ODEs with randomly  switching right-hand sides.
\end{abstract}


\section{Introduction}\label{section:introduction}
The primary motivation for this paper is to study parabolic partial
differential equations (PDEs) with randomly switching boundary
conditions. More precisely, given an elliptic differential operator,
$L$, on a domain $D\subset\mathbb{R}^{d}$, we want to study the
stochastic process $u(t,x)$ that solves $\partial_{t} u = L u$ in $D$
subject to boundary conditions that switch at random times between two
given boundary conditions. 

This type of random PDE is an example of a stochastic hybrid system.
The word ``hybrid'' is used because these stochastic processes
involve both continuous dynamics and discrete events. In this
example, the continuous dynamics are the different boundary value
problems corresponding to the different boundary conditions for the
given PDE, and the discrete events are when the boundary condition
switches. 

In general, a stochastic hybrid system is a continuous-time stochastic
process with two components: a continuous component $(X_{t})_{t\ge0}$
and a jump component $(J_{t})_{t\ge0}$. The jump component, $J_{t}$, is
a jump process on a finite set, and for each element of its state
space we assign some continuous dynamics to $X_{t}$. In between jumps
of $J_{t}$, the component $X_{t}$ evolves according to the dynamics
associated with the current state of $J_{t}$. When $J_{t}$ jumps, the
component $X_{t}$ switches to follow the dynamics associated with
the new state of $J_{t}$. 

An ordinary differential equation (ODE) with a 
switching right-hand side is the type of stochastic hybrid system that
is most commonly used in applications. Such ODE switching systems have
been used extensively in applied areas such as control theory,
computer science, and engineering (for example, \cite{Yin_2010},
\cite{boxma_/off_2005}, \cite{balde_note_2009}, and
\cite{lin09}). More recently, these systems have been used in diverse
areas of biology (for example, molecular biology
\cite{buckwar_exact_2011}, \cite{newby_asymptotic_2011},
\cite{bressloff_metastability_2013}, ecology
\cite{zhu_competitive_2009}, and epidemiology
\cite{farkas_pathogen_2011}). 
Furthermore, such ODE switching systems have also recently been the 
subject of much mathematical study (\cite{lawleyode},
\cite{hairer13}, \cite{benaim12quant}, \cite{benaim12qual},
\cite{bakhtin12}, \cite{hasler13finite}, \cite{hasler13asymptotic},
and \cite{belykh13}).

Comparatively, stochastic hybrid systems stemming from PDEs have
received little attention. While deterministic PDEs coupled to random
boundary conditions have been studied, the random boundary conditions
have 
typically been assumed to involve some Gaussian noise forcing
(\cite{bakhtin_burgers_2007}, \cite{duan_recent_2010},
\cite{wang_reductions_2009}, \cite{random_pde}, and
\cite{da_prato_1993}
).  The randomness enters our PDE system in a fundamentally different
way than in Stochastic PDEs which are driven by additive space-time
white noise (or even spatially smoother Gaussian fields). There the
fine scales are often asymptotically independent of each other
\cite{mattinglysuidan05,mattinglysuidan08}. Here, there is a single
piece of randomness which dictates the fine structure and hence the
fine scales, though not asymptotically deterministic, are
asymptotically perfectly correlated. (See Proposition~\ref{BC bounds}
for more details.)

 We were led to study such random PDEs by various biological applications. One
application is to insect respiration. Essentially all insects breathe via a network of
tubes that allow oxygen and carbon dioxide to diffuse to and from their cells \cite{wigglesworth_respiration_1931}.
Air enters and exits this network through valves (called spiracles) in the exoskeleton,
which regulate air 
flow by opening and closing quite irregularly in time.  This leads naturally to the following model problem.
Let $u(x,t)$ satisfy the heat equation $\partial_tu = D\partial_x^2u$ on $[0,L].$ $x = 0$ corresponds to tissue where the oxygen is absorbed, so $u(0,t) = 0.$  $x =L$ corresponds to a spiracle, so there the boundary condition switches between $\partial_x u(L,t) = 0$ (spiracle closed) and $u(L,t) = b > 0$ (spiracle open). Suppose the boundary conditions switch at exponential rates $r_0$ and $r_1$. We would like to calculate the long term statistics of the solution $u(t,x)$; in particular, we would would like to know how the oxygen absorption at the tissue, $\partial_x u(0,t)$  depends on $D$ and the switching rates. This model problem is fully developed in Section 4.3. 

A second such problem arises in understanding the concentration of neurotransmitters in the extracellular space in the brain. Imagine that the axonal projections from a nucleus of cells make a dense, random set of terminals in projection region, $P$, in the brain. For example, there is a dense set of terminals of serotonin neurons in the striatum that come from the dorsal raphe nucleus \cite{feldman}. Action potentials arrive as a Poisson process at a terminal and when they do,  neurotransmitter is released at a high rate into the extracellular space for a very short amount of time. At other times, the neurotransmitter is absorbed back into the terminal. We would like to calculate the long term statistics of the neurotransmitter concentration in the exterior domain that consists of $P$ with the terminal volumes removed. On the large scale, this is a homogenization problem. But to solve it, one first has to understand the local switching problem. The solution of the heat equation in the exterior domain, $u(x,t)$, satisfies $\partial_nu(x,t) = c >>0$ for a short time after an action potential has arrived, and $u(x,t) = 0$ at other times, for points $x$ on the boundary of a terminal. Thus, as in the previous paragraph we are switching between Dirichlet and Neuman boundary conditions at random times.  These questions are extremely important for neuroscience because it is now known that some groups of neurons affect distant locations of the brain by firing more or less and thus changing the ambient concentration of the neurotransmitter
in the extracellular space in the distant region, a phenomenon called ``volume transmission'' \cite{reed}\cite{fuxe}. An analysis of this problem, using the techniques developed in this paper, will be the subject of future work. 

Our paper is organized as follows. 
In Section~\ref{hilbertchapter}, we consider more general stochastic
hybrid systems from the viewpoint of iterated random functions (see
\cite{diaconis_iterated_1999} or \cite{Kifer86,Kifer88} for a review of iterated random
functions). Assuming that the continuous dynamics are contracting on
average, we prove convergence to a stationary distribution and
describe the structure and properties of this distribution. In
Section~\ref{pdechapter}, we apply these general results to the random
PDE problems described above. We show that the mean of the process
satisfies the PDE and that the mean of the stationary distribution
satisfies the time homogeneous version of the PDE. Then in
Section~\ref{section: examples}, we apply our results from
Sections~\ref{hilbertchapter} and \ref{pdechapter} to the
one-dimensional heat equation with randomly switching boundary
conditions. We find explicit formulae for various statistics of the
solution and obtain almost sure results about its regularity and
structure. There, we also show that our general results have
applications to other types of stochastic hybrid systems, such as ODEs
with randomly switching right-hand sides. Finally, we end
Section~\ref{section: examples} by explaining that our
results can be applied to the question in insect physiology mentioned above. 

We conclude this introduction by giving two examples that motivated our study. 
We return to these examples in Section~\ref{section: examples}.
Consider the heat equation on the interval $[0,L]$ with an absorbing boundary condition at 0 and a randomly switching boundary condition at $L$. Let the switching be controlled by a Markov jump process, $J_{t}$, on $\{0,1\}$ with $r_{0}$ and $r_{1}$ the respective rates for leaving states 0 and 1. In the following two examples, we consider different possible boundary conditions at $L$.

\begin{example}\label{ex: dir neu}
{\rm
Suppose the boundary condition at $L$ switches between an inhomogeneous Dirichlet condition and a Neumann no flux condition. More precisely, consider the stochastic process $u(t,x)\in L^{2}[0,L]$ that solves
\begin{align*}
\partial_{t}u & = D \Delta u \quad\text{in }(0,L)\\
u(0,t)  & = 0
\quad\text{and}\quad
J_{t}u_{x}(L,t) + (1-J_{t})(u(L,t)-b)  = 0.
\end{align*}
We show in Section~\ref{dir neu} that as $t\to\infty$, the process
$u(t,x)$ converges in distribution to an $L^{2}[0,L]$-valued random
variable whose expectation is a linear function. Letting
$\gamma=L\sqrt{(r_{0}+r_{1})/D }$ and $\rho=r_{0}/r_{1}$, we will show
that the
slope of this function is
\begin{align*}
\left(1+\frac{\rho}{\gamma}\tanh(\gamma)\right)^{-1}\frac{b}{L}.
\end{align*}}
\end{example}

\begin{example}\label{ex: dir dir}
{\rm
Suppose the boundary condition at $L$ switches between an inhomogeneous Dirichlet condition and a homogeneous Dirichlet condition. More precisely, consider the stochastic process $u(t,x)\in L^{2}[0,L]$ that solves
\begin{align*}
\partial_{t}u & = D \Delta u \quad\text{in }(0,L)\\
u(0,t) & = 0
\quad\text{and}\quad
J_{t}u(L,t) + (1-J_{t})(u(L,t)-b)  = 0.
\end{align*}
We show in Section~\ref{dir dir} that as $t\to\infty$, the process
$u(t,x)$ converges in distribution to an $L^{2}[0,L]$-valued random
variable whose expectation is a linear function. Letting
$p=r_{0}/(r_{0}+r_{1})$, we will show that the slope of this function is
\begin{align*}
(1-p)\frac{b}{L}.
\end{align*}}
\end{example}

The expectations for Examples~\ref{ex: dir neu} and \ref{ex: dir dir} are quite different.
In Example~\ref{ex: dir dir}, the expectation is the solution to the
time homogenous PDE with boundary conditions given by the average of
the two possible boundary conditions. We will see in Section~\ref{dir
  dir} that this simple result holds because the process switches
between boundary conditions of the same type and the corresponding semigroups commute. Moreover, because the
boundary conditions are the same type we will be able to compute
individual and joint statistics of the Fourier coefficients of the
stationary solution and show that this solution almost surely has a
very specific structure and regularity. 

In both examples, the expectation is a linear function with slope
given by $b/L$ multiplied by a factor less than one. While in
Example~\ref{ex: dir dir} this factor is simply the proportion of time
the boundary condition is inhomogeneous, the factor in
Example~\ref{ex: dir neu} is an unexpected expression involving the
hyperbolic tangent. Furthermore, while the factor in Example~\ref{ex:
  dir neu} still depends on the proportion of time the boundary
condition is inhomogeneous, it also depends on how often the boundary
conditions switch. Observe that if we keep this proportion fixed by
fixing the ratio $r_{0}/r_{1}$, and take the frequency of switches
small by letting $r_{0}+r_{1}$ go to 0, then the slope for
Example~\ref{ex: dir neu} approaches the same slope as in the
Example~\ref{ex: dir dir}. And if we keep the ratio $r_{0}/r_{1}$
fixed, but let the $r_{0}+r_{1}$ go to infinity, then the slope for
Example~\ref{ex: dir neu} approaches $b/L$. Some biological
implications of this result are discussed in
Section~\ref{section:insect}. 

\section{Abstract setting}\label{hilbertchapter}

We first consider stochastic hybrid systems in a separable Banach
space $X$. Under certain contractivity assumptions, we prove that the
process converges in distribution at large time and we show that the
limiting distribution satisfies certain invariance
properties. Although applicable to a range of stochastic hybrid
systems, the contents of this section will prove particularly useful
when we consider PDEs with randomly switching boundary conditions in
Sections~\ref{pdechapter} and \ref{section: examples}. 
\subsection{Discrete-time process}\label{section:discrete}
We first define the set $\Omega$ of all possible switching
environments and equip it with a probability measure $\P$ and
associated expectation $\E$. Let $\mu_{0}$
and $\mu_{1}$ be two probability distributions on the positive real
line. Define each switching environment, $\omega\in\Omega$, as
the 
sequence $\omega=(\omega_{1},\omega_{2},\dots)$, where each
$\omega_{k}$ is a pair of non-negative real numbers,
$(\tau^{k}_{0},\tau^{k}_{1})$, drawn from $\mu_{0}\times\mu_{1}$.
%\begin{revision}
That is, $(\tau^{k}_{0},\tau^{k}_{1})$ is an $\mathbb{R}^{2}$-valued random variable drawn from the product measure $\mu_{0}\times\mu_{1}$.
%\end{revision}
We
take $\P$ to be the infinite product measure generated by
$\mu_{0}\times\mu_{1}$. To
summarize some notation
\begin{align}\label{omega definition}
\omega = \left(\omega_{1},\omega_{2},\omega_{3},\dots\right) = \left((\tau^{1}_{0},\tau^{1}_{1}),(\tau^{2}_{0},\tau^{2}_{1}),(\tau^{3}_{0},\tau^{3}_{1}),\dots\right)\in \Omega.
\end{align}

For each $t\ge0$, let $\Phi_{t}^0(x)$ and $\Phi_{t}^1(x)$ be two
mappings from a separable Banach space $X$ to itself.  Make the
following assumptions on $\Phi^{i}_{t}$ for each $i\in\{0,1\}$,
$t\ge0$, $x,y\in X$, and with $\tau_{i}$ an independent draw from
$\mu_{i}$.
\begin{enumerate}[(a)]
\item
$\Phi_{t}^0(x)=x=\Phi_{t}^1(x)$ if $t=0$
\item
$t\mapsto \Phi^{i}_{t}(x)\in X$ is continuous
\item
$\E|\Phi^{i}_{\tau_{i}}(x)|<\infty$ 
\item
$|\Phi_{t}^{i}(x)-\Phi_{t}^{i}(y)|\le K_{i}(t)|x-y|$ for some $K_{i}(t)$
\item
$\E K_{0}(\tau_{1})\E K_{1}(\tau_{1})<\infty$ and 
$\E\log(K_{0}(\tau_{1})\E K_{1}(\tau_{1}))<0$
\end{enumerate}

For each $\omega\in\Omega$, $x\in X$, and natural number $k$, define the compositions
\begin{align*}
G_{\omega}^{k}(x)  := \Phi_{\tau^{k}_{1}}^1\circ \Phi_{\tau^{k}_{0}}^{0}(x)\qquad\text{and}\qquad
F_{\omega}^{k}(x)  := \Phi_{\tau^{k}_{0}}^0\circ \Phi_{\tau^{k}_{1}}^{1}(x).
\end{align*}
For each $\omega\in\Omega$, $x\in X$, and natural number $n>0$, we define the forward maps $\varphi^{n}$ and $\gamma^{n}$, and the backward maps $\varphi^{-n}$ and $\gamma^{-n}$ by the following compositions of $G$ and $F$:
\begin{align}\label{discrete process}
\begin{split}
\varphi^{n}_{\omega}(x) & = G_{\omega}^{n}\circ \cdots \circ G_{\omega}^{1}(x)
\qquad\text{and}\qquad
\gamma^{n}_{\omega}(x)  = F_{\omega}^{n}\circ \cdots \circ F_{\omega}^{1}(x),\\
\varphi^{-n}_{\omega}(x) & = G_{\omega}^{1}\circ \cdots \circ G_{\omega}^{n}(x)
\qquad\text{and}\qquad
\gamma^{-n}_{\omega}(x)  = F_{\omega}^{1}\circ \cdots \circ F_{\omega}^{n}(x).
\end{split}
\end{align}
To make our notation complete, we define $\varphi^{0}(x) = x = \gamma^{0}(x)$.

\begin{remark}\label{discrete remark}
{\rm The maps $\varphi^{n}$ and $\gamma^{n}$ are iterated random functions, (see \cite{diaconis_iterated_1999} for a review).
Assumptions (d) and (e) above ensure that $G^{k}$ and $F^{k}$ are contracting on average.
Thus, $\{\varphi^{n}\}_{n\ge0}$ and $\{\gamma^{n}\}_{n\ge0}$ are Markov chains with invariant probability distributions given by the distributions of the almost sure limits of $\varphi^{-n}$ and $\gamma^{-n}$ as $n\to\infty$, respectively. Moreover, the distributions of the Markov chains $\varphi^{n}$ and $\gamma^{n}$ converge at a geometric rate to these invariant distributions. These results are immediately attained by applying theorems in, for example, \cite{diaconis_iterated_1999,Kifer88,Duflo}. Nonetheless, we prove the following proposition to make our results more self-contained.}
\end{remark}

\begin{proposition}\label{prop:pullback}
Define
\begin{align}\label{pullback 0}
Y_{1}(\omega)  := \lim_{n\to\infty}\varphi^{-n}_{\omega}(x)
\quad\text{and}\quad
Y_{0}(\omega)  := \lim_{n\to\infty}\gamma^{-n}_{\omega}(x).
\end{align}
These limits exist almost surely and are independent of $x \in
X$.
\end{proposition}
\begin{remark}
{\rm
  A random set which attracts all initial data started at
  ``$-\infty$'' and is forward-invariant under the dynamics is called
  a random pullback attractor \cite{Kifer88,CrauelFlandoli94,Crauel01, Mattingly02}.
  When that attractor consists of a single point almost surely then it
  is called a random point attractor. In this case, the single point
  can be viewed as a random variable.  Random variables such as these
  are often called random pullback attractors, or ``pullbacks'' for
  short, because they take an initial condition $x$ and pull it back
  to the infinite past \cite{Crauel01,Schmalfuss96,Mattingly99}. Since when the
  random attractor is a single point one can associate to each
  realization of random ``forcing'' a single attracting solution which
  gives the asymptotic behavior, this is also ofter referred to as the
  ``one force, one solution''
  paradigm \cite{EKhaninMazelSinai,Mattingly99,Mattingly02}.}
\end{remark}
\begin{proof}
We will show that the sequence $\varphi^{-n}(x)$ is almost surely Cauchy. Let $x_1,x_2\in X$ and $n\ge m$. Using the triangle inequality repeatedly, we obtain
\begin{align}\label{triangle repeat}
\begin{split}
  |\varphi^{-n}(x_{1})-\varphi^{-m}(x_{2})|
& \le \sum_{i=m+1}^n |G^{1}\circ\dots\circ G^{i}(x_1)-G^{1}\circ\dots\circ G^{i-1}(x_2)|\\
& \le \sum_{i=m+1}^n|G^{i}(x_1)-x_2|\bigg(\prod_{j=1}^{i-1}K_{0}(\tau_{0}^{j})K_{1}(\tau_{1}^{j})\bigg).
\end{split}
\end{align}
Assumptions (c), (d), and (e) give that $\E|G^{i}(x_1)-x_2|<\infty$, and thus a simple application of the Borel Cantelli lemma gives the existence of an almost surely finite random constant $C_1(\omega)$ such that for all $i\in\mathbb{N}$
\begin{align}\label{borel c 1}
|G^{i}(x_1)-x_2|<C_1(\omega)(i)^2.
\end{align}

Let $Z_j:=K_{0}(\tau_{0}^{j})K_{1}(\tau_{1}^{j})$. A standard argument (see, for example, \cite{diaconis_iterated_1999}, Lemmas~5.2 and 5.4) gives the existence of constants $\epsilon>0$, $A>0$, and $0<r<1$ such that for all $i\in\mathbb{Z}$, we have 
$\P(\sum_{j=1}^i\log Z_j>-i\epsilon)<A r^i$, by assumption (e). 
Thus, another application of the Borel Cantelli lemma gives the existence of an almost surely finite random constant $C_2(\omega)$ such that for all $i\in\mathbb{N}$
\begin{align}\label{borel c 2}
\sum_{j=1}^i\log Z_j\le-i\epsilon+C_2(\omega).
\end{align}
Plugging the bounds from equations~(\ref{borel c 1}) and (\ref{borel c 2}) into equation~(\ref{triangle repeat}) gives
\begin{align*}
  |\varphi^{-n}(x_{1})-\varphi^{-m}(x_{2})|
  \le \sum_{i=m+1}^nC_1(\omega)(i)^2e^{-i\epsilon+C_2(\omega)}.
  \end{align*}
Therefore $\varphi^{-n}(x_{1})$ is almost surely Cauchy and thus $Y_{1}$ exists almost surely. 
Since $x_1$ and $x_2$ were arbitrary, $Y_{1}$ is independent of the $x$ used in its definition. The proof for $Y_{0}$ is similar.
\qquad\end{proof}

The random variables $Y_{1}$ and $Y_{0}$ satisfy the following invariance properties.
\begin{proposition}\label{invariance}
Let $\tau_{0}$ and $\tau_{1}$ be independent draws from $\mu_{0}$ and $\mu_{1}$. Then
\begin{align}\label{eq:invariance}
Y_{0}  \eqdist\Phi^{0}_{\tau_{0}}(Y_{1})
\quad\text{and}\quad
Y_{1} \eqdist \Phi^{1}_{\tau_{1}}(Y_{0})
\end{align}
where $\eqdist$ denotes equal in distribution.
\end{proposition}

\begin{proof}
Let $y\in X$ and observe that for any $n\in\mathbb{N}$, we have that
\begin{align*}
\gamma^{-n}(y) \eqdist \Phi^{0}_{\tau_{0}}\left(\varphi^{n-1}(\Phi^{1}_{\tau_{1}}(y))\right).
\end{align*}
Taking the limit as $n\to\infty$ yields
\begin{align}\label{take limit}
\lim_{n\to\infty}\gamma^{-n}(y) \eqdist \lim_{n\to\infty}\Phi^{0}_{\tau_{0}}\left(\varphi^{n-1}(\Phi^{1}_{\tau_{1}}(y))\right)
= \Phi^{0}_{\tau_{0}}\left(\lim_{n\to\infty}\varphi^{n-1}(\Phi^{1}_{\tau_{1}}(y))\right)
\end{align}
since $\Phi_{t}^{0}(x)$ is continuous in $x$ for each $t$. Recalling that the definitions of $Y_{0}$ and $Y_{1}$ in Equation~(\ref{pullback 0}) are independent of $x$ by Proposition~\ref{prop:pullback}, we have that Equation~(\ref{take limit}) becomes $Y_{0}\eqdist\Phi^{0}_{\tau_{0}}(Y_{1})$. The proof that $Y_{1} \eqdist \Phi^{1}_{\tau_{1}}(Y_{0})$ is similar.
\qquad\end{proof}

\begin{proposition}\label{prop:invariant set}
Suppose there exists a nonempty 
set $S\subset X$ so that for all $t\ge0$, $\Phi^{i}_{t}:S\to S$ for $i=0$ and 1. Then $Y_{0}$ and $Y_{1}$ are in the closure, $\bar{S}$, almost surely.
\end{proposition}
\begin{proof}
If $x\in S$, then $\varphi^{-n}(x)\in S$ almost surely for all $n\ge0$ by assumption. Thus, $\lim_{n\to\infty}\varphi^{-n}(x) = Y_{1}\in \bar{S}$ almost surely. But by Proposition~\ref{prop:pullback}, the random variable $Y_{1}$ is independent of the initial $x$ used in its definition, so $Y_{1}\in \bar{S}$ almost surely. The proof for $Y_{0}$ is similar.
\qquad\end{proof}

%%%%%%%%%%%%%%%%%%%%%%%%%%%%%%%%%%%%%%%%%%%%%%%%%%%%%%%%%%%%%%%

\subsection{Continuous-time process}\label{section:cts}
To define the continuous time process, we need more notation. Much of the following notation is standard in renewal theory.
For each $\omega\in\Omega$ and natural number $n$, define
\begin{align*}
S_{n} & := \displaystyle\sum_{k=1}^{n}\big(\tau^{k}_{0}+\tau^{k}_{1}\big)
\end{align*}
with $S_{0}:=0$. Define $S'_{n+1} := S_{n}+\tau^{n+1}_{0}$ for $n\ge0$. Observe that $S'_{n}<S_{n}<S'_{n+1}<S_{n+1}$ by definition. Define
\begin{align*}
N_{t} & := \max\{n\ge0:S_{n}\le t\}.
\end{align*}
We also define the state process $J_{t}$ for $t\ge0$ by
\begin{align}\label{z defn}
J_{t} & := \begin{cases}
0 & \text{$S_{N_{t}}\le t< S'_{N_{t}+1}$}\\
1 & \text{$S'_{N_{t}+1}\le t$}.
\end{cases}
\end{align}
Finally, for $t\ge0$, define the elapsed time since the last switch, often called the age process, by
\begin{align*}
a_{t} & :=  J_{t}(t-S'_{N_{t}+1}) + (1-J_{t})(t-S_{N_{t}}).
\end{align*}

We are now ready to define our continuous-time $X$-valued process. For $u_{0}\in X$, $\omega\in\Omega$, and $t\ge0$, define
\begin{align}\label{process definition}
u(t,\omega) & = J_{t}\Phi^{1}_{a_{t}}\circ \Phi^{0}_{\tau^{N_{t}+1}_{0}}(\varphi^{N_{t}}(u_{0})) + (1-J_{t})\Phi^{0}_{a_{t}}(\varphi^{N_{t}}(u_{0})).
\end{align}

%%%%%%%%%%%%%%%%%%%%%%%%%%%%%%%%%%%%%%%%%%%%%%%%%%%%%%%%%%%%%%%

\subsection{Convergence in distribution to mixture of pullbacks}\label{section:conv}
In this section we will find the limiting distribution of $u(t)$ as $t\to\infty$. In order to describe this limiting distribution, we will need to define three more random variables.
Define $a^{0}$ and $a^{1}$ to be two random variables with the following cumulative distribution functions:
\begin{align*}
\P(a^{0}\le x) & = 
\frac{\E\min(\tau_{0},x)}{\E\tau_{0}}
\quad\text{and}\quad
\P(a^{1}\le x)  = 
\frac{\E\min(\tau_{1},x)}{\E\tau_{1}}.
\end{align*}
We will see in Lemma~\ref{age convergence} that the distributions of
$a^{0}$ and $a^{1}$ can be thought of as the limiting distributions of
the age process conditioned on either $J_{t}=0$ or 1.  Let $\xi$ be a
Bernoulli random variable with parameter
$p:=(\E\tau_{1})/(\E\tau_{0}+\E\tau_{1})$, the probability that
$J_{t}=1$ at large time. Assume $a^{0}$, $a^{1}$, and $\xi$ are all
chosen to be mutually independent and independent of
$(\tau_{0}^{k},\tau_{1}^{k})$ for every $k$.  Recall that a measure
$\mu$ on the real line is said to be arithmetic if there exist a $d>0$
so that $\mu(\{0,d,2d,\dots\})=1$.

\begin{theorem}\label{conv in dist}
  Suppose $\Phi^{0}_{t}$ and $\Phi^{1}_{t}$ satisfy assumptions
  (a)-(e) of Section~\ref{section:discrete}. Let $u(t)$ be defined as
  in Equation~(\ref{process definition}), and $a^{0}$, $a^{1}$, and
  $\xi$ as in the above paragraph. If the switching time
  distributions, $\mu_{0}$ and $\mu_{1}$, are non-arithmetic, then we
  have the following convergence in distribution as $t\to\infty$.
%\begin{revision}
\begin{align*}
(u(t),J_t) \indist (\bar{u},\xi)\quad\text{as }t\to\infty,
\end{align*}
where $\bar{u}:=\xi\Phi^{1}_{a^{1}}(Y_{0}) + (1-\xi)\Phi^{0}_{a^{0}}(Y_{1})$.
%\end{revision}
\end{theorem}

%\begin{revision}
The pullbacks $Y_0$ and $Y_1$ give the invariant distributions of the discrete Markov processes $\varphi^n$ and $\gamma^n$ defined in Equation~(\ref{discrete process}) (see Proposition~\ref{prop:pullback} and Remark~\ref{discrete remark}). Thus, Theorem~\ref{conv in dist} describes the limiting distribution of the (not necessarily Markovian) continuous process $(u(t),J_t)$ in terms of the invariant distributions of related Markov processes. Stated colloquially, Theorem~\ref{conv in dist} says that to go from the invariant distributions of the discrete Markov processes to the limiting distribution of the continuous process, one must do the following: first flip a coin with parameter $p$ to decide which map ($\Phi^0$ or $\Phi^1$) is being applied and then apply that map (say it's $\Phi^i$) to pullback $Y_{1-i}$ for time $a^i$, where $a^i$ is the amount of time since the last switch given that $\Phi^i$ is currently being applied.

In the context of piecewise deterministic Markov processes given by switching ordinary differential equations, the authors of \cite{benaim12qual} relate the invariant measure of the continuous Markov process to the imbedded discrete Markov chain. In particular, in Proposition 2.4 of \cite{benaim12qual}, the authors show that the ergodic probability measures of the continuous and discrete processes form a one to one correspondence and hence the continuous process is stable if and only if the discrete process is stable.

In the following corollary, we show that if the switching time distributions are exponential so that the continuous process $(u(t),J_t)$ is Markov, then its limiting distribution actually is the invariant distribution of the embedded discrete Markov process (if the initial map is either $\Phi^1$ or $\Phi^0$ with probability $p$ or $1-p$). The corollary follows immediately from Proposition~\ref{invariance} and Theorem~\ref{conv in dist} since the age of a Poisson process is exponentially distributed.

%\end{revision}
\begin{corollary}\label{cor conv}
  Suppose the switching time distributions, $\mu_{0}$ and $\mu_{1}$,
  are exponential with respective rate parameters $r_{0}$ and
  $r_{1}$. If $\xi$ is Bernoulli with parameter $r_{0}/(r_{0}+r_{1})$,
  then we have the following convergence in distribution as
  $t\to\infty$.
\begin{align*}
%\begin{revision}
(u(t),J_t) \indist (\bar{u},\xi)\quad\text{as }t\to\infty,
\end{align*}
where $\bar{u}:=\xi Y_{1} + (1-\xi)Y_{0}$.
%\end{revision}
\end{corollary}

{\em Proof of Theorem~\ref{conv in dist}}.
%\begin{revision}
In light of Proposition~\ref{invariance}, it is enough to prove the desired convergence in distribution for $\bar{u}:=\Phi^{1}_{a^{1}}\circ \Phi^{0}_{\tau_{0}}(Y_{1}) + (1-\xi)\Phi^{0}_{a^{0}}(Y_{1})$,
%\end{revision}
where $\tau_{0}$ is an independent draw from $\mu_{0}$.
%\begin{revision}
Recall that if $Z$ is a random variable taking values in some metric space and a Borel subset $S$ of that metric space satisfies $\P(Z\in\partial S)=0$ where $\partial S$ denotes the boundary of $S$, then $S$ is called a continuity set of $Z$. Let $A$, $B$, $C$, and $D$ be continuity sets of $\xi a^{1}+(1-\xi)a^{0}$, $\xi\tau_{0}$, $\xi$, and $Y_{1}$, respectively.
%\end{revision}
We will show that
\begin{multline}\label{convergence on rectangles}
\P(a_{t}\in A, J_{t}\tau^{N_{t}+1}_{0}\in B, J_{t}\in C, \varphi^{N_{t}}(u_{0})\in D)\\
 \to \P(\xi a^{1}+(1-\xi)a^{0}\in A, \xi\tau_{0}\in B, \xi\in C, Y_{1}\in D)\quad\text{as}\quad t\rightarrow\infty.
\end{multline}
%\begin{revision}

Once this convergence is shown, the conclusion of the theorem quickly follows. To see this, assume the convergence in Equation~(\ref{convergence on rectangles}) holds. Define the $(\mathbb{R}^{3}\times X)$-valued random variable $Y_{t}:=(a_{t},J_{t}\tau^{N_{t}+1}_{0},J_{t},\varphi^{N_{t}})$, where we have suppressed the $u_0$ dependence. We will usually suppress this dependence since the limiting random variables don't depend on the initial $u_0$ (see Proposition~\ref{prop:pullback}). Since $X$ is assumed to be separable, the product $\mathbb{R}^{3}\times X$ is separable and thus we can apply Theorem 2.8 in \cite{billingsley_convergence_1999} to obtain that $Y_{t}$ converges in distribution to $(\xi a^{1}+(1-\xi)a^{0}, \xi\tau_{0}, \xi, Y_{1})$ as $t\to\infty$.

Define the function $g:\mathbb{R}^{3}\times X\to X$ by $g(a,t,j,y) = j\Phi^{1}_{a}\circ \Phi^{0}_{t}(y)+(1-j)\Phi^{0}_{t}(y)$ and observe that $u(t) = g(a_{t},\tau_{0}^{N_{t+1}},J_{t},\varphi^{N_{t}})$ and $\bar{u}=g(a^{1},\tau_{0},\xi,Y)$. Further define the function $h:(\mathbb{R}^{3}\times X)\to (X\times\mathbb{R})$ by $h(a,t,j,y)=(g(a,t,j,y),j)$.
The function $h$ is continuous because $g$ is continuous and thus the conclusion of the theorem follows from the continuous mapping theorem (see, for example, Theorem 3.2.4 in \cite{durrett_probability_2010}). Therefore, it remains only to show the convergence in Equation~(\ref{convergence on rectangles}).
Our proof relies on the auxiliary Lemmas~\ref{age convergence}-\ref{asymp indep3}, whose proofs are given in Subsection~\ref{section:lemmas}. 
%\end{revision}

In what follows, we will make extensive use of indicator functions. For ease of reading, we will often denote the indicator $1_{A}=1_{A}(\omega)$ by $\{A\}=\{A\}(\omega)$.

For each $t\ge0$, define $\mathcal{F}_{t}$ to be the $\sigma$-algebra generated by $S_{N_{t}}$ and $\displaystyle\{(\tau^{k}_{0},\tau^{k}_{1})\}_{k=N_{t}+1}^{\infty}$. Since $a_{t}$, $\tau_{0}^{N_{t}+1}$, and $J_{t}$ are measurable with respect to $\mathcal{F}_{t}$, the tower property of conditional expectation and the triangle inequality give
\begin{align*}
 &\Big|\E\{a_{t}\in A,J_{t}\tau_{0}^{N_{t}+1}\in B,J_{t}\in C,\varphi^{N_{t}}\in D\} \\ &\qquad\qquad\qquad- \E\{\xi a^{1}+(1-\xi)a^{0}\in A,\xi\tau_{0}\in B,\xi\in C,Y_{1}\in D\}\Big|\\
 &\qquad \le  \Big|\E\left[\{a_{t}\in A,J_{t}\tau_{0}^{N_{t}+1}\in B,J_{t}\in C\}\E [\{\varphi^{N_{t}}\in D\}|\mathcal{F}_{t}]\right] \\
 &\qquad\qquad\qquad- \E\left[\{a_{t}\in A,J_{t}\tau_{0}^{N_{t}+1}\in B,J_{t}\in C\}\E [\{Y_{1}\in D\}]\right]\Big|\\
 &\qquad\qquad\qquad\qquad\qquad+ \Big|\E\left[\{a_{t}\in A,J_{t}\tau_{0}^{N_{t}+1}\in B,J_{t}\in C\}\E [\{Y_{1}\in D\}]\right] \\
 &\qquad\qquad\qquad\qquad\qquad\qquad- \E\{\xi a^{1}+(1-\xi)a^{0}\in A,\xi\tau_{0}\in B,\xi\in C\}\{Y_{1}\in D\}\Big|.
\end{align*}
By Lemma~\ref{asymp indep}, we have that $\E \left[\{\varphi^{N_{t}}\in D\}|\mathcal{F}_{t}\right]\to \E \left[\{Y_{1}\in D\}\right]$ almost surely as $t\to\infty$. Therefore, the first term goes to 0 by the dominated convergence theorem. Since $Y_{1}$ is independent of $\xi$, $a^{1}$, $a^{0}$, and $\tau_{0}$, the second term is bounded above by
\begin{align}\label{second term}
\Psi:=\big|\E \{a_{t}\in A,J_{t}\tau_{0}^{N_{t}+1}\in B,J_{t}\in C\} - \E \{\xi a^{1}+(1-\xi)a^{0}\in A,\xi\tau_{0}\in B,\xi\in C\}\big|.
\end{align}
To show that $\Psi$ goes to 0 as $t\to\infty$, we consider the four possible cases for the inclusion of 0 and 1 in $C$. If both 0 and 1 are not in $C$, then $\Psi$ is 0 for all $t\ge0$ since $J_{t}$ and $\xi$ are each almost surely 0 or 1.

Suppose $0\in C$ and $1\notin C$. Then the indicator function in the
first term of $\Psi$ is only non-zero if $J_{t}=0$. Hence, we can
replace $\{J_{t}\in C\}$ by $(1-J_{t})$ and
$\{J_{t}\tau_{0}^{N_{t}+1}\in B\}$ by $\{0\in B\}$. Similarly, in the
second term we replace $\{\xi\in C\}$ by $(1-\xi)$, $\{\xi\tau_{0}\in
B\}$ by $\{0\in B\}$, and $\{\xi a^{1}+(1-\xi)a^{0}\in A\}$ by
$\{a^{0}\in A\}$. Thus $\Psi$ becomes
\begin{align*}
\Psi
& = |\E \{a_{t}\in A,0\in B\}(1-J_{t}) - \E \{a^{0}\in A,0\in B\}(1-\xi)|\\
& \le |\E \{a_{t}\in A\}(1-J_{t}) - \E \{a^{0}\in A\}(1-\xi)|.
\end{align*}
By Lemma~\ref{age convergence}, this term goes to 0 as $t\to\infty$.

Suppose $1\in C$ and $0\notin C$. Then the indicator function in the
first term of $\Psi$ is only non-zero if $J_{t}=1$. Thus after
performing similar replacements to those above, $\Psi$ becomes
\begin{align*}
\Psi
& = |\E \{a_{t}\in A,\tau_{0}^{N_{t}+1}\in B\}J_{t} - \E \{a^{1}\in A,\tau_{0}\in B\}\xi|.
\end{align*}
Define $\mathcal{F}'_{t}$ to be the $\sigma$-algebra generated by $S'_{N_{t}+1}$, $\tau^{N_{t}+1}_{1}$, and $\displaystyle\{(\tau^{k}_{0},\tau^{k}_{1})\}_{k=N_{t}+2}^{\infty}$. Observe that $J_{t}$ and $a_{t}$ are both measurable with respect to $\mathcal{F}'_{t}$. Therefore, by the tower property of conditional expectation and the triangle inequality we have that
\begin{multline*}
 |\E \{a_{t}\in A,\tau_{0}^{N_{t}+1}\in B\}J_{t} - \E \{a^{1}\in A,\tau_{0}\in B\}\xi|\\
 \le |\E \left[\{a_{t}\in A\}J_{t}\E \left[\{\tau_{0}^{N_{t}+1}\in B\}|\mathcal{F}'_{t}\right]\right] - \E \left[\{a_{t}\in A\}J_{t}\E \left[\{\tau_{0}\in B\}\right]\right]|\\
 + |\E [\{a_{t}\in A\}J_{t}\E [\{\tau_{0}\in B\}]] - \E \{a^{1}\in A,\tau_{0}\in B\}\xi|.
\end{multline*}
Lemma~\ref{asymp indep2} gives us that $J_{t}\E [\{\tau_{0}^{N_{t}+1}\in B\}|\mathcal{F'}_{t}] = J_{t}\E [\{\tau_{0}^{1}\in B\}|\mathcal{F'}_{t}]$ almost surely and Lemma~\ref{asymp indep3} gives that $\E [\{\tau_{0}^{1}\in B\}|\mathcal{F'}_{t}]\to \E [\{\tau_{0}\in B\}]$ almost surely as $t\to\infty$. Therefore, the first term goes to 0 as $t\to\infty$ by the dominated convergence theorem.
Finally since $\tau_{0}$ is independent of $\xi$ and $a^{1}$, we have the following bound on the second term
\begin{equation*}
 |\E [\{a_{t}\in A\}J_{t}\E [\{\tau_{0}\in B\}]]-\E \{a^{1}\in A,\tau_{0}\in B\}\xi| \le |\E \{a_{t}\in A\}J_{t}-\E \{a^{1}\in A\}\xi|
\end{equation*}
This goes to 0 as $t\to\infty$ by Lemma~\ref{age convergence}.

Finally, if both $0\in C$ and $1\in C$, then $\Psi$ becomes
\begin{align*}
\Psi
 = &|\E \{a_{t}\in A,J_{t}\tau_{0}^{N_{t}+1}\in B\} - \E \{\xi a^{1}+(1-\xi)a^{0}\in A,\xi\tau_{0}\in B\}|\\
 \le &|\E \{a_{t}\in A,\tau_{0}^{N_{t}+1}\in B\}J_{t} - \E \{a^{1}\in A,\tau_{0}\in B\}\xi|\\
 &\qquad+ |\E \{a_{t}\in A,0\in B\}(1-J_{t}) - \E \{a^{0}\in A,0\in B\}(1-\xi)|.
\end{align*}
We've already shown that each of these terms go to zero as $t\to\infty$, so the proof is complete.\qquad\endproof

\subsection{The Lemmas}\label{section:lemmas}

We now state and prove all of the lemmas that are needed for Theorem~\ref{conv in dist}. This first lemma calculates the limiting distribution of the age process. It can be interpreted as first flipping a coin to determine if $J_{t}$ is 0 or 1, and then choosing from the limiting distribution of the age conditioned on $J_{t}$.
%\begin{revision}
\begin{lemma}\label{age convergence}
For any continuity set $A$ of $\xi a^{1}+(1-\xi)a^{0}$ we have that as $t\to\infty$
\begin{align*}
|\E \{a_{t}\in A\}J_{t}-\E \{a^{1}\in A\}\xi| + |\E \{a_{t}\in A\}(1-J_{t})-\E \{a^{0}\in A\}(1-\xi)|\to0.
\end{align*}
\end{lemma}
\begin{proof}
We will first show the desired convergence for sets of a special form and then extend to a continuity set. %end{revision}
Let $x\ge0$ and consider the alternating renewal process that is said to be ``on'' when $0\le t-S'_{N_{t}+1}\le x$ and ``off'' otherwise. Formally, we define the ``on/off'' state process
\begin{align*}
b_{t} & = \begin{cases}
1 \quad&\text{if }0\le t-S'_{N_{t}+1}\le x\\
0 \quad&\text{otherwise}.
\end{cases}
\end{align*}
Observe
that 
the lengths of time that the process is ``on'' are
$\{\min(\tau^{k}_{1},x)\}_{k=1}^{\infty}$. Similarly, the lengths of
time that the process is ``off'' are $\tau^{1}_{0}$ and
$\{\tau^{k}_{0}+(\tau^{k-1}_{1}-x)^{+}\}_{k=2}^{\infty}$, where as
usual $(y)^{+}$ is equal to $y$ if $y\ge0$ and 0 otherwise. Since the
distribution of
$\min(\tau^{k}_{1},x)+\tau^{k}_{0}+(\tau^{k-1}_{1}-x)^{+}$ is
nonarithmetic, and since $\E
[\min(\tau^{k}_{1},x)+\tau^{k}_{0}+(\tau^{k-1}_{1}-x)^{+}]<\infty$, we
can apply Theorem 3.4.4 in \cite{ross_stochastic_1996} to obtain
\begin{align}\label{on/off convergence}
  \lim_{t\to\infty}\P(b_{t}=1) = \frac{\E \min(\tau_{1},x)}{\E
    [\min(\tau^{k}_{1},x)+\tau^{k}_{0}+(\tau^{k-1}_{1}-x)^{+}]}.
\end{align}
Informally, this intuitive result states that the probability that the
alternating renewal process is ``on'' at large time is just the
expected length of an ``on'' bout divided by the sum of the expected
lengths of an ``off'' bout and an ``on'' bout. Since $\E
[\min(\tau^{k}_{1},x)+\tau^{k}_{0}+(\tau^{k-1}_{1}-x)^{+}]=\E
\tau_{0}+\E \tau_{1}$ and since the distribution of $a^{1}$ is chosen
so that $\E \tau_{1}\P(a^{1}\le x)=\E \min(\tau_{1},x)$,
Equation~(\ref{on/off convergence}) simplifies to
\begin{align*}
\lim_{t\to\infty}\P(b_{t}=1)
=  \frac{\E \tau_{1}\P(a^{1}\le x)}{\E \tau_{0}+\E \tau_{1}}.
\end{align*}

Therefore
\begin{equation}\label{hard convergence}
\begin{array}{r@{}l}
\E [\{a_{t}\le x\}J_{t}] & = \P(a_{t}\le x,J_{t}=1)
 = \P(0\le t-S_{N_{t}+1}\le x)\\
&= \P(b_{t}=1)
 \xrightarrow[t\to\infty]{} \displaystyle\frac{\E \tau_{1}\P(a^{1}\le x)}{\E \tau_{0}+\E \tau_{1}}
 = \E [\{a^{1}\le x\}\xi].
\end{array}
\end{equation}
The last equality holds because $\xi$ and $a^{1}$ are independent and $\E \xi=\E \tau_{1}/(\E \tau_{0}+\E \tau_{1})$.

%\begin{revision}
Further, since the switching time distributions, $\mu_0$ and $\mu_1$, are non-arithmetic, we can apply Theorem 3.4.4 in \cite{ross_stochastic_1996} to conclude that
\begin{align}\label{easy convergence}
|\P(J_t=0)-\P(\xi=0)|\to0\quad\text{as }t\to\infty.
\end{align}
Now observe that
\begin{align*}
|\E\{J_ta_t\le x\}-\E\{\xi a^1\le x\}|\le |\P(J_t=0)-\P(\xi=0)|+|\E\{a_t\le x\}J_t-\E\{a^1\le x\}\xi |.
\end{align*}
This bound and the convergence in Equations~(\ref{hard convergence}) and (\ref{easy convergence}) gives that $|\E\{J_ta_t\le x\}-\E\{\xi a^1\le x\}|\to0$ as $t\to\infty$. Thus, $J_ta_t\indist\xi a^1$. Further, it's easy to see that any continuity set of $\xi a^{1}+(1-\xi)a^{0}$ must be a continuity set of $\xi a^{1}$. Thus, for any continuity set $A$ of $\xi a^{1}+(1-\xi)a^{0}$ we have that $|\E \{J_{t}a_{t}\in A\}-\E \{\xi a^{1}\in A\}|\to0$ by the Portmanteau theorem (see, for example, \cite{billingsley_convergence_1999}).  Thus,
\begin{align*}
|\E &\{a_{t}\in A\}J_{t} -\E \{a^{1}\in A\}\xi|\\
& \le |\E \{J_{t}a_{t}\in A\}-\E \{\xi a^{1}\in A\}|+|\P(J_t=0)-\P(\xi=0)| \to 0\quad\text{as }t\to\infty.
\end{align*}
%\end{revision}

The analogous argument shows that $|\E \{a_{t}\in A\}(1-J_{t})-\E \{a^{0}\in A\}(1-\xi)|\to0$ as $t\to\infty$ and the proof is complete.
\qquad\end{proof}

The next three lemmas are general results that are all relatively standard. We return to lemmas specific to our problem in Lemma~\ref{phi order}.

\begin{lemma}\label{backwards convergence}
  Suppose $X_{t}\to X_{\infty}$ a.s. as $t\to\infty$ and $|X_{t}|\le
  B$ a.s. where $B$ is a random variable satisfying $\E B<\infty$. If
  $\mathcal{F}_{t}\subset\mathcal{F}_{s}$ for $0\le s\le t$ is a
  right-contiuous filtration, and
  $\mathcal{F}_{\infty}:=\cap_{t\ge0}\mathcal{F}_{t}$, then
\begin{align*}
\E [X_{t}|\mathcal{F}_{t}]\to \E [X_{\infty}|\mathcal{F}_{\infty}]\quad\text{almost surely as $t\to\infty$.}
\end{align*}
\end{lemma}

\begin{proof}
  We first show the convergence for an $X_\infty=X_{t}$ independent of
  $t$. Let $X_\infty$ be any integrable random variable and for $t\le0$
  define
\begin{align*}
M_{t}:=\E \left[X_\infty|\mathcal{F}_{-t}\right].
\end{align*}
We claim that $\{M_{t}\}_{t=0}^{-\infty}$ is a backwards martingale
with respect to the filtration $\{\mathcal{M}\}_{t=0}^{-\infty}$ where
$\mathcal{M}_{t}:=\mathcal{F}_{-t}$. For $s\le t\le0$ we have that
$\mathcal{F}_{-s}\subset\mathcal{F}_{-t}$ and therefore by the tower
property of conditional expectation,
\begin{align*}
  \E [M_{t}|\mathcal{F}_{-s}]=\E \left[\E
    \left[X_\infty|\mathcal{F}_{-t}\right]|\mathcal{F}_{-s}\right] = \E
  \left[X_\infty|\mathcal{F}_{-s}\right] = M_{s}.
\end{align*}
Since by definition of conditional expectation
$M_{t}\in\mathcal{F}_{-t}$, and since $M_{t}\le B$ almost surely where
$\E B<\infty$, we have that $M_{t}$ is indeed a backwards
martingale. Because of the continuity properties of $\mathcal{F}_{t}$
we know we have a separable version of $M_t$. By the backwards
martingale convergence theorem (see \cite{doob1953}[Theorem 4.2s, p. 354] for example),
$M_{-\infty}:=\lim_{t\to-\infty}M_{t}$ exists almost surely and in
$L^{1}(\Omega)$.

We claim that $M_{-\infty}=\E
\left[X_\infty|\mathcal{F}_{\infty}\right]$. Since for $t\le T\le0$ we have
that $M_{t}\in\mathcal{F}_{-t}\subset\mathcal{F}_{-T}$, it follows
that $M_{-\infty}\in\mathcal{F}_{-T}$. Since $T\le0$ was arbitrary,
$M_{-\infty}\in\mathcal{F}_{\infty}$.

Let $A\in\mathcal{F}_{\infty}$. Then
\begin{align*}
  |\E M_{-t}1_{A}-\E M_{-\infty}1_{A}| \le \E
  |M_{-t}1_{A}-M_{-\infty}1_{A}| \le \E |M_{-t}-M_{-\infty}| \to
  0\quad\text{as }t\to\infty
\end{align*}
since $M_{-t}\to M_{-\infty}$ in $L^{1}(\Omega)$. But,
\begin{align*}
\E M_{-t}1_{A} = \E [\E \left[X_\infty|\mathcal{F}_{t}\right]1_{A}] = \E [\E \left[X_\infty1_{A}|\mathcal{F}_{t}\right]] = \E X_\infty1_{A}.
\end{align*}
Therefore $\E X1_{A}=\E M_{-\infty}1_{A}$, and so we conclude that
$M_{-\infty}=\E \left[X_\infty|\mathcal{F}_{\infty}\right]$.

We now show the convergence for the case where $X_{t}$ depends on $t$.
Let $T\ge0$ and define $B_{T}:=\sup\{|X_{t}-X_{s}|:t,s>T\}$. $B_{T}\le
2B$, so $B_{T}$ is integrable. Thus,
\begin{align*}
  \limsup_{t\to\infty}\E
  \left[|X_{t}-X_{\infty}|\mathcal{F}_{t}\right]\le\lim_{t\to\infty}\E
  \left[B_{T}|\mathcal{F}_{t}\right] & = \E
  \left[B_{T}|\mathcal{F}_{\infty}\right]
\end{align*}
By assumption, $B_{T}\to0$ a.s. as $T\to\infty$ so by Jensen's
inequality
\begin{align*}
  |\E \left[X_{t}|\mathcal{F}_{t}\right]-\E
  \left[X_{\infty}|\mathcal{F}_{t}\right]| & \le \E
  \left[|X_{t}-X_{\infty}\|\mathcal{F}_{t}\right]\to0.
\end{align*}
Therefore,
\begin{align*}
|\E \left[X_{t}|\mathcal{F}_{t}\right]-\E \left[X_{\infty}|\mathcal{F}_{\infty}\right]|
 \le |\E \left[X_{t}|\mathcal{F}_{t}\right]-\E \left[X_{\infty}|\mathcal{F}_{t}\right]|
 + |\E \left[X_{\infty}|\mathcal{F}_{t}\right]-\E \left[X_{\infty}|\mathcal{F}_{\infty}\right]|.
\end{align*}
We've just shown that the first term goes to 0, and we've shown that the second term goes to 0 since $X_{\infty}$ doesn't depend on $t$, so the proof is complete.
\qquad\end{proof}

\begin{lemma}\label{random indexing}
If $X_{n}\to X_{\infty}$ a.s. as $n\to\infty$ and $N_{t}\to\infty$ a.s. as $t\to\infty$, then
\begin{align*}
X_{N_{t}}\to X_{\infty}\quad\text{a.s. as }t\to\infty.
\end{align*}
\end{lemma}
%\begin{proof}
{\em Proof}. 
Let $A:=\{X_{n}\nrightarrow X_{\infty}\}$ and $B:=\{N_{t}\nrightarrow \infty\}$. Then
\begin{align*}
\P(X_{N_{t}}\nrightarrow X_{\infty})\le \P(A\cup B)\le \P(A) + \P(B) = 0.
\qquad\endproof
\end{align*}
%\qquad\end{proof}

We now give some standard definitions. Let $(\Omega,\mathcal{F},P)$ be a probability space. A measurable map $\pi:\Omega\to\Omega$ is said to be \textbf{measure preserving} if $\P(\pi^{-1}A)=\P(A)$ for all $A\in\mathcal{F}$. Let $\pi$ be a given measure preserving map. A set $A\in\mathcal{F}$ is said to be \textbf{$\pi$-invariant} if $\pi^{-1}A=A$, where two sets are considered to be equal if their symmetric difference has probability 0. A random variable $X$ is said to be \textbf{$\pi$-invariant} if $X=X\circ\pi$ almost surely.

\begin{lemma}\label{invariant sets}
Let $\pi:\Omega\to\Omega$ be a measure preserving map. If $X$ is $\pi$-invariant, then so is every set in its $\sigma$-algebra.
\end{lemma}
\begin{proof}
See, for example, \cite{durrett_probability_2010} Exercise 7.1.1.
\qquad\end{proof}

%\begin{revision}
We remind the reader that we suppress the $u_0$ dependence and write $\varphi^{N_t}(u_0)=\varphi^{N_t}$. We will usually suppress this dependence since the limiting random variables don't depend on the initial $u_0$ (see Proposition~\ref{prop:pullback}).
%\end{revision}

\begin{lemma}\label{phi order}
For each $t\ge0$ define $\mathcal{F}_{t}$ to be the $\sigma$-algebra generated by $S_{N_{t}}$ and $\displaystyle\{(\tau^{k}_{0},\tau^{k}_{1})\}_{k=N_{t}+1}^{\infty}$. If $D$ is a Borel set of $X$, then for each $t\ge0$
\begin{align*}
\E \left[\{\varphi^{N_{t}}\in D\}|\mathcal{F}_{t}\right] = \E \left[\{\varphi^{-N_{t}}\in D\}|\mathcal{F}_{t}\right]\quad\text{a.s.}
\end{align*}
\end{lemma}

\begin{remark}
{\rm
To see why this lemma should be true, observe that (a) the random variables $\varphi^{N_{t}}$ and $\varphi^{-N_{t}}$ are equal after a re-ordering of the first $N_{t}$-many $\omega_{k}$'s and that (b) the random variables generating $\mathcal{F}_{t}$ don't depend on the order of the first $N_{t}$-many $\omega_{k}$'s.}
\end{remark}

\begin{proof}
Fix a $t\ge0$ and let $A\in\mathcal{F}_{t}$. By the definition of conditional expectation, we have that
\begin{align*}
\int_{\Omega}\E \left[\{\varphi^{N_{t}}\in D\}|\mathcal{F}_{t}\right](\omega)\{A\}(\omega)\,d\P
& = \int_{\Omega}\{\varphi^{N_{t}}\in D\}(\omega)\{A\}(\omega)\,d\P.
\end{align*}

Define $\sigma_{t}:\Omega\to\Omega$ to be the permutation that inverts the order of the first $N_{t}$-many $\omega_{k}$'s. That is, $(\sigma_{t}(\omega))_{k}=\omega_{N_{t}-k+1}$ for $k\in\{1,\dots,N_{t}\}$ and $(\sigma_{t}(\omega))_{k}=\omega_{k}$ for $k>N_{t}$.
Observe that $N_{t}(\omega)=N_{t}(\sigma_{t}(\omega))$ and thus $\varphi^{N_{t}}(\omega) = \varphi^{-N_{t}}(\sigma_{t}(\omega))$. Also, $S_{N_{t}}$ and $\displaystyle\{(\tau^{k}_{0},\tau^{k}_{1})\}_{k=N_{t}+1}^{\infty}$ are $\sigma_{t}$-invariant, so $A$ is $\sigma_{t}$-invariant by Lemma~\ref{invariant sets}. Thus
\begin{align*}
\int_{\Omega}\{\varphi^{N_{t}}\in D\}(\omega)\{A\}(\omega)\,d\P
& = \int_{\Omega}\{\varphi^{-N_{t}}\in D\}(\sigma_{t}(\omega))\{A\}(\sigma_{t}(\omega))\,d\P.
\end{align*}
Since $\sigma_{t}$ is measure preserving and by the definition of conditional expectation,
\begin{align*}
\int_{\Omega}\{\varphi^{-N_{t}}\in D\}(\sigma_{t}(\omega))\{A\}(\sigma_{t}(\omega))\,d\P
& = \int_{\Omega}\{\varphi^{-N_{t}}\in D\}(\omega)\{A\}(\omega)\,d\P\\
& = \int_{\Omega}\E \left[\{\varphi^{-N_{t}}\in D\}|\mathcal{F}_{t}\right](\omega)\{A\}(\omega)\,d\P.
\end{align*}
Putting this all together,
\begin{align*}
\int_{\Omega}\E [\{\varphi^{N_{t}}\in D\}|\mathcal{F}_{t}]\{A\}\,d\P
& =\int_{\Omega}\E [\{\varphi^{-N_{t}}\in D\}|\mathcal{F}_{t}]\{A\}\,d\P.
\end{align*}
Since $A$ was an arbitrary element of $\mathcal{F}_{t}$, the proof is complete.
\qquad\end{proof}

Recall that the random variable $Y_{1}$ is defined by $Y_{1}:=\lim_{n\to\infty}\varphi^{-n}(x)$, and is independent of the choice of $x\in X$, by Proposition~\ref{prop:pullback}.

\begin{lemma}\label{asymp indep}
For each $t\ge0$ define $\mathcal{F}_{t}$ to be the $\sigma$-algebra generated by $S_{N_{t}}$ and $\displaystyle\{(\tau^{k}_{0},\tau^{k}_{1})\}_{k=N_{t}+1}^{\infty}$. If $D$ is a $Y_1$-continuity set, then with probability one
\begin{align}\label{f}
\E \left[\{\varphi^{-N_{t}}\in D\}|\mathcal{F}_{t}\right] \to \E \{Y_{1}\in D\}\quad\text{as }t\to\infty\\
\text{and}\quad \E \left[\{\varphi^{N_{t}}\in D\}|\mathcal{F}_{t}\right] \to \E \{Y_{1}\in D\}\quad\text{as }t\to\infty
\end{align}
\end{lemma}

\begin{proof}
In light of Lemma~\ref{phi order}, it suffices to show the convergence in Equation~(\ref{f}).

Since $\varphi^{-n}\to Y_{1}$ almost surely as $n\to\infty$ and since $N_{t}\to\infty$ almost surely as $t\to\infty$, we have that $\varphi^{-N_{t}}\to Y_{1}$ almost surely by Lemma~\ref{random indexing}. %\begin{revision}
We claim that $\{\varphi^{-N_{t}}\in D\}\to\{Y_{1}\in D\}$ almost surely. Since $\varphi^{-N_{t}}\to Y_{1}$ almost surely and $\P(Y_1\in\partial D)=0$, there exists a set $S\subset\Omega$ of full measure such that if $\omega\in S$ then $\varphi^{-N_{t}}(\omega)\to Y_{1}(\omega)\notin \partial D$ as $t\to\infty$. Let $\omega\in S$. If $Y_1(\omega)\in D$, then $Y_1(\omega)$ must be in the interior of $D$ and hence there exists some $r>0$ such that the ball of radius $r$ centered at $Y_1(\omega)$ is contained in the interior of $D$. Since $\varphi^{-N_{t}}(\omega)\to Y_{1}(\omega)$, there must exist some $T$ such that $\varphi^{-N_{t}}(\omega)$ is within $r$ of $Y_1(\omega)$ for all $t\ge T$ and hence $\varphi^{-N_{t}}(\omega)\in D$ for all $t\ge T$. Thus, $\{\varphi^{-N_{t}}\in D\}(\omega)\to\{Y_{1}\in D\}(\omega)=1$. A similar argument shows that $\{\varphi^{-N_{t}}\in D\}(\omega)\to\{Y_{1}\in D\}(\omega)=0$ if $\omega\in S$ is such that $Y_1(\omega)\notin D$. Hence, $\{\varphi^{-N_{t}}\in D\}\to\{Y_{1}\in D\}$ on $S$ which has full measure so the convergence is almost sure.
%\end{revision}

Define $\mathcal{F}_{\infty}:=\cap_{t\ge0}\mathcal{F}_{t}$ and observe that $\mathcal{F}_{t}\subset\mathcal{F}_{s}$ for $t\ge s\ge0$. Thus, by Lemma~\ref{backwards convergence}, 
\begin{align*}
\E \left[\{\varphi^{-N_{t}}\in D\}|\mathcal{F}_{t}\right]\to \E \left[\{Y_{1}\in D\}|\mathcal{F}_{\infty}\right]\quad\text{almost surely as }t\to\infty.
\end{align*}

To complete the proof, we will show that for every $A\in\mathcal{F}_{\infty}$, $\P(A)=0$ or 1. To show this, we will show that $\mathcal{F}_{\infty}$ is contained in the exchangeable $\sigma$-algebra and then apply the Hewitt-Savage zero-one law. Let $n\in\mathbb{N}$, $A\in\mathcal{F}_{\infty}$, and $\pi_{n}$ be an arbitrary permutation of $\omega_{1},\dots,\omega_{n}$.
Define $\pi_{t}:\Omega\to\Omega$ by
\begin{align*}
\pi_{t}(\omega) & = \begin{cases}
\pi_{n}(\omega)\quad & N_{t}\ge n\\
\omega\quad & N_{t}<n.
\end{cases}
\end{align*}
Since $S_{N_{t}}$ and $\displaystyle\{(\tau^{k}_{0},\tau^{k}_{1})\}_{k=N_{t}+1}^{\infty}$ are $\pi_{t}$-invariant, then $A$ is $\pi_{t}$-invariant by Lemma~\ref{invariant sets} as $A\in\mathcal{F}_{\infty}\subset\mathcal{F}_{t}$.
Therefore
\begin{align*}
\P(A\Delta\pi^{-1}_{n}A,N_{t}\ge n)
 = \P(A\Delta\pi^{-1}_{t}A,N_{t}\ge n)
 \le \P(A\Delta\pi^{-1}_{t}A)
 = 0.
\end{align*}
Hence
\begin{align*}
\P(A\Delta\pi^{-1}_{n}A) & = \P(A\Delta\pi^{-1}_{n}A,N_{t}\ge n)+\P(A\Delta\pi^{-1}_{n}A,N_{t}< n)\\
& \le \P(A\Delta\pi^{-1}_{n}A,N_{t}< n)\le \P(N_{t}< n).
\end{align*}

Since $t$ was arbitrary, and because $\P(N_{t}< n)\to0$ as $t\to\infty$ since $N_{t}\to\infty$ almost surely, we conclude that $\P(A\Delta\pi^{-1}_{n}A)=0$. Since $\pi_{n}$ was an arbitrary finite permutation, we conclude that $\mathcal{F}_{\infty}$ is contained in the exchangeable $\sigma$-algebra. By the Hewitt-Savage zero-one law, $\mathcal{F}_{\infty}$ only contains events that have probability 0 or 1. Thus, $\{Y_{1}\in D\}$ is trivially independent of $\mathcal{F}_{\infty}$ and therefore $\E \left[\{Y_{1}\in D\}|\mathcal{F}_{\infty}\right] = \E \{Y_{1}\in D\}$.
\qquad\end{proof}

\begin{lemma}\label{asymp indep2}
For each $t\ge0$, define $\mathcal{F}'_{t}$ to be the $\sigma$-algebra generated by $S'_{N_{t}+1}$, $\tau^{N_{t}+1}_{1}$, and $\displaystyle\{(\tau^{k}_{0},\tau^{k}_{1})\}_{k=N_{t}+2}^{\infty}$. Then
\begin{align*}
J_{t}\E \left[\{\tau_{0}^{N_{t}+1}\in B\}|\mathcal{F'}_{t}\right] = J_{t}\E \left[\{\tau_{0}^{1}\in B\}|\mathcal{F'}_{t}\right]\quad\text{almost surely.}
\end{align*}
\end{lemma}

\begin{remark}
{\rm
Recall that $J_{t}$ is either 0 if $S_{N_{t}}\le t<S'_{N_{t}+1}$ or 1 if $S'_{N_{t}+1}\le t$. Hence, this Lemma states that $\E \left[\{\tau_{0}^{N_{t}+1}\in B\}|\mathcal{F'}_{t}\right] = \E \left[\{\tau_{0}^{1}\in B\}|\mathcal{F'}_{t}\right]$ if $J_{t}=1$.}
\end{remark}
\begin{remark}
{\rm
The proof of this Lemma is very similar to the proof of Lemma~\ref{phi order}.}
\end{remark}

\begin{proof}
If $\omega$ is such that $J_{t}=0$, then the equality is trivially satisfied.
Let $A\in\mathcal{F}'_{t}$. Since $\{\omega\in\Omega:J_{t}(\omega)=1\}\in\mathcal{F}'_{t}$, we have by the definition of conditional expectation that
\begin{align*}
& \int_{\Omega}\E \left[\{\tau_{0}^{N_{t}+1}\in B\}|\mathcal{F}'_{t}\right]\{A,J_{t}=1\}\,d\P 
 = \int_{\Omega}\{\tau_{0}^{N_{t}+1}\in B\}(\omega)\{A,J_{t}=1\}(\omega)\,d\P.
\end{align*}
Define $\sigma_{t}:\Omega\to\Omega$ by
\begin{align*}
(\sigma_{t}(\omega))_{k} & = \begin{cases}
(\tau_{0}^{N_{t}+1},\tau^{1}_{1})\quad & \text{if $k=1$ and $J_{t}=1$}\\
(\tau_{0}^{1},\tau^{N_{t}+1}_{1})\quad & \text{if $k=N_{t}+1$ and $J_{t}=1$}\\
\omega_{k}\quad & \text{otherwise}.
\end{cases}
\end{align*}
That is, $\sigma_{t}$ switches $\tau_{0}^{1}$ and $\tau_{0}^{N_{t}+1}$ if $J_{t}=1$ and  otherwise does nothing. Since $S'_{N_{t}+1}$, $\tau^{N_{t}+1}_{1}$, and $\displaystyle\{(\tau^{k}_{0},\tau^{k}_{1})\}_{k=N_{t}+2}^{\infty}$ are all $\sigma_{t}$-invariant, we have that $A$ is $\sigma_{t}$-invariant by Lemma~\ref{invariant sets}. Also observe that $\{J_{t}=1\}$ is $\sigma_{t}$-invariant. Thus
\begin{align*}
\int_{\Omega}\{\tau_{0}^{N_{t}+1}\in B\}(\omega)\{A,J_{t}=1\}(\omega)\,d\P
& = \int_{\Omega}\{\tau_{0}^{1}\in B\}(\sigma_{t}(\omega))\{A,J_{t}=1\}(\sigma_{t}(\omega))\,d\P.
\end{align*}
Since $\sigma_{t}$ is measure preserving, and by the definition of conditional expectation, we have that
\begin{align*}
\int_{\Omega}\{\tau_{0}^{1}\in B\}(\sigma_{t}(\omega))\{A,J_{t}=1\}(\sigma_{t}(\omega))\,d\P
& = \int_{\Omega}\{\tau_{0}^{1}\in B\}(\omega)\{A,J_{t}=1\}(\omega)\,d\P\\
& = \int_{\Omega}\E \left[\{\tau_{0}^{1}\in B\}|\mathcal{F}'_{t}\right]\{A,J_{t}=1\}\,d\P.
\end{align*}
Putting all this together,
\begin{align*}
\int_{\Omega}\E [\{\tau_{0}^{N_{t}+1}\in B\}|\mathcal{F}'_{t}]\{A,J_{t}=1\}\,d\P
& = \int_{\Omega}\E [\{\tau_{0}^{1}\in B\}|\mathcal{F}'_{t}]\{A,J_{t}=1\}\,d\P.
\end{align*}

This implies that $\E [\{\tau_{0}^{N_{t}+1}\in B\}|\mathcal{F'}_{t}] = \E [\{\tau_{0}^{1}\in B\}|\mathcal{F'}_{t}]$ almost surely on $\{J_{t}=1\}$. To see this, let $\epsilon>0$ define $\Lambda:=\{\omega\in\Omega:\E [\{\tau_{0}^{N_{t}+1}\in B\}|\mathcal{F'}_{t}]-\E [\{\tau_{0}^{1}\in B\}|\mathcal{F'}_{t}]\ge\epsilon\}$. This set is in $\mathcal{F}'_{t}$, so by the above calculation we have that
\begin{align*}
0 = \int_{\Lambda\cap\{J_{t}=1\}}\E [\{\tau_{0}^{N_{t}+1}\in B\}|\mathcal{F'}_{t}]-\E [\{\tau_{0}^{1}\in B\}|\mathcal{F'}_{t}]\,d\P
\ge \epsilon \P(\Lambda\cap\{J_{t}=1\}).
\end{align*}
So $\P(\Lambda\cap\{J_{t}=1\})=0$. The same argument with $\Lambda':=\{\omega\in\Omega:\E [\{\tau_{0}^{1}\in B\}|\mathcal{F}'_{t}]-\E [\{\tau_{0}^{N_{t}+1}\in B\}|\mathcal{F}'_{t}]\ge\epsilon\}$ completes the proof of the claim. Therefore, $J_{t}\E [\{\tau_{0}^{N_{t}+1}\in B\}|\mathcal{F}'_{t}] = J_{t}\E [\{\tau_{0}^{1}\in B\}|\mathcal{F}'_{t}]$ almost surely.
\qquad\end{proof}

\begin{lemma}\label{asymp indep3}
For each $t\ge0$, define $\mathcal{F}'_{t}$ to be the $\sigma$-algebra generated by $S'_{N_{t}+1}$, $\tau^{N_{t}+1}_{1}$, and $\displaystyle\{(\tau^{k}_{0},\tau^{k}_{1})\}_{k=N_{t}+2}^{\infty}$. Then
\begin{align*}
\E \left[\{\tau_{0}^{1}\in B\}|\mathcal{F'}_{t}\right]\to \E \{\tau_{0}\in B\}\quad\text{almost surely as }t\to\infty.
\end{align*}
\end{lemma}

\begin{remark}
{\rm The proof of this Lemma is very similar to the proof of Lemma~\ref{asymp indep}.}
\end{remark}

\begin{proof}
Define $\mathcal{F'}_{\infty}:=\cap_{t\ge0}\mathcal{F}'_{t}$ and observe that $\mathcal{F'}_{s}\supset\mathcal{F'}_{t}$ for $0\le s\le t$. Thus, by Lemma~\ref{backwards convergence} 
\begin{align*}
\E \left[\{\tau_{0}^{1}\in B\}|\mathcal{F'}_{t}\right]\to \E \left[\{\tau_{0}^{1}\in B\}|\mathcal{F}'_{\infty}\right]\quad\text{almost surely.}
\end{align*}

We claim that for each $A\in\mathcal{F}'_{\infty}$, $\P(A)=0$ or 1. To show this, we will show that $\mathcal{F}'_{\infty}$ is contained in the exchangeable $\sigma$-algebra and then apply the Hewitt-Savage zero-one law. Let $n\in\mathbb{N}$, $A\in\mathcal{F}'_{\infty}$, and $\pi_{n}$ be an arbitrary permutation of $(\tau_{0}^{1},\tau_{1}^{1}),\dots,(\tau_{0}^{n},\tau_{1}^{n})$.

Define $\pi_{t}:\Omega\to\Omega$ by
\begin{align*}
\pi_{t}(\omega) & = \begin{cases}
\pi_{n}(\omega)\quad & N_{t}\ge n\\
\omega\quad & N_{t}<n.
\end{cases}
\end{align*}
Since $S'_{N_{t}+1}$, $\tau^{N_{t}+1}_{1}$, and $\displaystyle\{(\tau^{k}_{0},\tau^{k}_{1})\}_{k=N_{t}+2}^{\infty}$ are $\pi_{t}$-invariant, then $A$ is $\pi_{t}$-invariant by Lemma~\ref{invariant sets} as $A\in\mathcal{F}'_{\infty}\subset\mathcal{F}'_{t}$.
Therefore
\begin{align*}
\P(A\Delta\pi^{-1}_{n}A,N_{t}\ge n)
 = \P(A\Delta\pi^{-1}_{t}A,N_{t}\ge n)
 \le \P(A\Delta\pi^{-1}_{t}A)
 = 0.
\end{align*}
Hence
\begin{align*}
\P(A\Delta\pi^{-1}_{n}A) & = \P(A\Delta\pi^{-1}_{n}A,N_{t}\ge n)+\P(A\Delta\pi^{-1}_{n}A,N_{t}< n)\\
& \le \P(A\Delta\pi^{-1}_{n}A,N_{t}< n) \le \P(N_{t}< n).
\end{align*}
Since $t$ was arbitrary, we conclude that $\P(A\Delta\pi^{-1}_{n}A)=0$ because $\P(N_{t}< n+1)\to0$ as $t\to\infty$ since $N_{t}\to\infty$ almost surely. Since $\pi_{n}$ was an arbitrary finite permutation, we conclude that $\mathcal{F}'_{\infty}$ is contained in the exchangeable $\sigma$-algebra. By the Hewitt-Savage zero-one law, $\mathcal{F}'_{\infty}$ contains only events that have probability 0 or 1. Thus, $\{\tau_{0}^{1}\in C\}$ is trivially independent of $\mathcal{F}'_{\infty}$ and so we conclude that $\E \left[\{\tau_{0}^{1}\in C\}|\mathcal{F}'_{\infty}\right] = \E \{\tau_{0}^{1}\in C\} = \E \{\tau_{0}\in C\}$.
\qquad\end{proof}

%%%%%%%%%%%%%%%%%%%%%%%%%%%%%%%%%%%%%%%%%%%%%%%%%%%%%%%%%%%%%%%

\section{PDEs with randomly switching boundary conditions}\label{pdechapter}

We now use our results from Section~\ref{hilbertchapter} to study parabolic PDEs with randomly switching boundary conditions. Our results apply to a range of specific problems, so in Section \ref{general examples} we explain how to cast a problem in our framework. 
 In Section~\ref{pde setup} we 
collect assumptions and in Section~\ref{mean} we prove theorems about the mean of the process.

%%%%%%%%%%%%%%%%%%%%%%%%%%%%%%%%%%%%%%%%%%%%%%%%%%%%%%%%%%%%%

\subsection{General setup}\label{general examples}

Our results can be applied to the following type of random PDE. Suppose we are given a strongly elliptic, symmetric, linear, second order differential operator $L$ on a domain $D\subset\mathbb{R}^{d}$ with smooth coefficients which do not depend on $t$. Assume the domain $D$ is bounded with a smooth boundary.  We consider the stochastic process $u(t,x)$ that solves 
\begin{align}\label{pde}
\partial_{t} u = L u \quad\text{in }D
\end{align}
subject to boundary conditions that switch at random times between two given boundary conditions, $(a)$ and $(b)$. We allow $(a)$ and $(b)$ to be different types; for example, one can be Dirichlet and the other Neumann. 
For the sake of presentation, we assume $(a)$ are homogenous, but our analysis is easily modified to include the case where $(a)$ are inhomogenous.

We formulate this problem in the setting of Section~\ref{hilbertchapter} as alternating flows on the Hilbert space $L^{2}(D)$. We define
\begin{align*}
Au  := L u \quad\text{if $u\in D(A)$}
\qquad\text{and}\qquad
Bu  := L u \quad\text{if $u\in D(B)$}
\end{align*}
where $D(A)$ is chosen so that $A$ generates the contraction $C_{0}$-semigroup that maps an initial condition to the solution of Equation~(\ref{pde}) at time $t$ subject to boundary conditions $(a)$, and $D(B)$ is chosen so that $B$ generates the contraction $C_{0}$-semigroup that maps an initial condition to the solution of Equation~(\ref{pde}) at time $t$ subject to the homogenous version of boundary conditions $(b)$. We then choose $h(t):[0,\infty)\to D(L)$ to satisfy $\partial_{t}h=Lh$ with boundary conditions $(b)$ and initial condition $h(0)=0$. Then
the $H$-valued process defined in Equation~(\ref{process definition}) in Section~\ref{hilbertchapter} with $\Phi^{1}_{t}(f) = e^{At}f$ and $\Phi^{0}_{t}(f) = e^{Bt}f+h(t)$ corresponds to this random PDE.

%%%%%%%%%%%%%%%%%%%%%%%%%%%%%%%%%%%%%%%%%%%%%%%%%%%%%%%%%%%%

\subsection{Assumptions}\label{pde setup}

We now formalize the setup from Section~\ref{general examples}. Let $H$ be a real separable Hilbert space with inner product $\<\cdot,\cdot\>$ and let $A$ and $B$ be two self-adjoint operators on $H$, one with strictly negative spectrum and one with non-positive spectrum. Hence, $A$ and $B$ generate contraction $C_{0}$-semigroups, which we denote respectively by $e^{At}$ and $e^{Bt}$. Assume $A=B$ on $D(A)\cap D(B)\ne\emptyset$. Assume there exists a continuous function $h(t):[0,\infty)\to H$ satisfying $h(0)=0$ and $\frac{d}{dt}\<\phi,h(t)\>=\<B\phi,h(t)\>$ for all $\phi\in D(A)\cap D(B)$.
Recalling notation from Section~\ref{section:discrete}, let the switching time distributions, $\mu_{0}$ and $\mu_{1}$, be continuous distributions on the positive real line.

Let $u(t,\omega)$ be the $H$-valued process defined in Equation~(\ref{process definition}) in Section~\ref{section:cts} with
\begin{align*}
\Phi^{1}_{t}(f)  = e^{At}f
\qquad\text{and}\qquad
\Phi^{0}_{t}(f)  = e^{Bt}f + h(t). 
\end{align*}
It's easy to check that $\Phi^{1}_{t}$ and $\Phi^{0}_{t}$ satisfy assumptions (a)-(e) from Section~\ref{section:discrete}.

Assume there exists a deterministic $M=M(u_{0})$ so that with probability one, $\|u(t)\|\le M$ for each $t\ge0$, where $\|x\|:=\sqrt{\<x,x\>}$.

For every $0<s\le t$, define $\eta(s,t)$ to be the random variable that gives the number of switches that occur on the interval $(s,t)$. Formally, we define $\eta(s,t)$ by taking the supremum over partitions $\sigma$ of the interval $(s,t)$, $s=\sigma_{0}<\sigma_{1}<\dots<\sigma_{k}<\sigma_{k+1}=t$,
\begin{align}\label{eta definition}
\eta(s,t)(\omega):= \sup_{\sigma}\sum_{i=0}^{k}|J_{\sigma_{i+1}}(\omega)-J_{\sigma_{i}}(\omega)|,
\end{align}
where $J_{t}$ is as in Equation~(\ref{z defn}). Assume that $\mu_{0}$ and $\mu_{1}$ are such that for every $t>0$, we have that as $s\to0$,
\begin{align*}
\P\left(\eta(t,t+s)=1\right) & =O(s)
\quad\text{and}\quad
\P\left(\eta(t,t+s)\ge2\right)  =o(s).
\end{align*}

%%%%%%%%%%%%%%%%%%%%%%%%%%%%%%%%%%%%%%%%%%%%%%%%%%%%%%%%%%%

\subsection{The mean satisfies the PDE}\label{mean}
In what follows, fix $\phi\in D(A)\cap D(B)$, which will serve as our test function. 
The following theorem states that the mean of our process satisfies the weak form of the PDE.
%\begin{revision}
We note that we use $\E$ to denote the Bochner integral of Hilbert space-valued random variables.
%\end{revision}
\begin{theorem}\label{exppde}
For each $\phi\in D(A)\cap D(B)$ and $t>0$, we have that
\begin{align*}
\frac{d}{dt}\<\phi,\E u(t)\> & = \<A\phi,\E u(t)\>.
\end{align*}
\end{theorem}
To prove this theorem, we need a few lemmas. Our first lemma states that each realization our stochastic process satisfies the weak form of the PDE away from switching times.

%%%%%%%%%%%%%%%%%%%%%%%%%%%%%%%%%%%%%%%%%%%%%%%%%%%%%%%%%%
\begin{lemma}\label{differentiate}
Let $\omega_{0}\in\Omega$ be given. If $t_{0}>0$ is such that $t_{0}\ne S_{k}(\omega_{0})$ and $t_{0}\ne S'_{k}(\omega_{0})$ for every $k$, then for all $t$ in some neighborhood of $t_{0}$,
\begin{align*}
\frac{d}{dt}\<\phi,u(t,\omega_{0})\> & = \<A\phi,u(t,\omega_{0})\>.
\end{align*}
\end{lemma}
{\em Proof}.
By the definition of $u(t,\omega)$ and the assumption that $A$ and $B$ are self-adjoint, we can write the inner product of $\phi$ and $u$ as
\begin{align*}
\<\phi,u(t)\>
& = \<\phi,e^{Aa_{t}}u(S'_{N_{t}+1})\>J_{t} + \<\phi,e^{Ba_{t}}u(S_{N_{t}})+h(a_{t})\>(1-J_{t})\\
& = \<e^{Aa_{t}}\phi,u(S'_{N_{t}+1})\>J_{t} + \<e^{Ba_{t}}\phi,u(S_{N_{t}})\>(1-J_{t})+\<\phi,h(a_{t})\>(1-J_{t}).
\end{align*}

We now calculate $\frac{d}{dt}e^{Aa_{t}}\phi$ and $\frac{d}{dt}e^{Ba_{t}}\phi$ where $\frac{d}{dt}$ means
the limit in $H$ of the difference quotients. Since $t_{0}$ is such that $t_{0}\ne S_{k}(\omega_{0})$ and $t_{0}\ne S'_{k}(\omega_{0})$ for all $k$, there exists a neighborhood $J(\omega_{0})=J$ of $t_{0}$ so that no switches occur in $J$. Therefore $S_{N_{t}}$, $S'_{N_{t}+1}$, and $J_{t}$ are constant on $J$. And since $e^{At}$ is a $C_{0}$-semigroup and $\phi\in D(A)$, we have that for all $t\in J$
\begin{align*}
\frac{d}{dt}e^{Aa_{t}}\phi  = \frac{d}{dt}e^{A(t-S'_{N_{t}+1})}\phi
 = \frac{d}{dt}e^{At}e^{-AS'_{N_{t}+1}}\phi
 = Ae^{At}e^{-AS'_{N_{t}+1}}\phi
 = Ae^{Aa_{t}}\phi.
\end{align*}
Similarly, 
$\frac{d}{dt}e^{Ba_{t}}\phi = Be^{Ba_{t}}\phi.$
Since strongly convergent sequences in $H$ are weakly convergent, and again since $S_{N_{t}}$, $S'_{N_{t}+1}$, and $J_{t}$ are constant on $J$, we have that for all $t\in J$
\begin{align*}
\frac{d}{dt}\<\phi,u(t)\>
& = \<Ae^{Aa_{t}}\phi,u(S'_{N_{t}+1})\>J_{t} + \big(\<Be^{Ba_{t}}\phi,u(S_{N_{t}})\>+\frac{d}{dt}\<\phi,h(a_{t})\>\big)(1-J_{t}).
\end{align*}
Since $A$ and $B$ are self-adjoint, $A=B$ on $D(A)\cap D(B)$, and $\frac{d}{dt}\<\phi,h(t)\> = (B\phi,h(t))$, we conclude that for all $t\in J$,
\begin{align*}
\frac{d}{dt}\<\phi,u(t)\>
& = \<A\phi,e^{Aa_{t}}u(S'_{N_{t}+1})\>J_{t} + \<B\phi,e^{Ba_{t}}u(S_{N_{t}})+h(a_{t})\>(1-J_{t})\\
& = \<A\phi,u(t)\>J_{t} + \<B\phi,u(t)\>(1-J_{t}) = \<A\phi,u(t)\>.
\qquad\endproof
\end{align*}

The next lemma states that our process satisfies a weak continuity condition.

\begin{lemma}\label{unif cts2}
For every $\epsilon>0$ and $t>0$, there exists a $\delta(\epsilon,t)>0$ so that if $|t-s|<\delta(\epsilon,t)$, then
\begin{align*}
|\<\phi,u(t,\omega)-u(s,\omega)\>1_{\eta(s,t)=1}|<\epsilon\quad\text{a.s.}
\end{align*}
\end{lemma}
{\em Proof}.
Let $s$ and $t$ be given and let $\rho$ be the minimum of $s$ and $t$. Observe that if there are no switches between $s$ and $t$ and $J_{\rho}=0$, then
\begin{align*}
\left|\<\phi,u(t,\omega)-u(s,\omega)\>\right|
 = \left|\<\phi,[e^{A|t-s|}-I]u(\rho,\omega)\>\right|
 \le \|e^{A|t-s|}\phi-\phi\|M,
\end{align*}
since $A$ is self-adjoint and $\|u(t)\|\le M$ a.s. by assumption.
Similarly, suppose there are no switches between $s$ and $t$ and $J_{\rho}=1$. If $M_{2}=\max_{\xi\le 2t}\|h(\xi)\|$ and $|t-s|<t$, then we have by the mean value theorem
\begin{align*}
\big|\<\phi,u(t,\omega)&-u(s,\omega)\>\big|
 \le \big|\<\phi,[e^{B|t-s|}-I][u(\rho,\omega)-h(a_{\rho})]\>\big| + \big|\<\phi,h(a_{t})-h(a_{s})\>\big|\\
& \le \big|\<[e^{B|t-s|}-I]\phi,u(\rho,\omega)-h(a_{\rho})\>\big| + |t-s|\max_{\xi\le 2t}\frac{d}{dt}\big|\<\phi,h(\xi)\>\big|\\
& \le \|e^{B|t-s|}\phi-\phi\|\big(M+M_{2}\big) + |t-s|\|B\phi\|M_{2}
\end{align*}

Since $e^{At}$ and $e^{Bt}$ are both $C_{0}$-semigroups, we can choose a $0<\delta(\epsilon,t)<t$ so that if $|t-s|<\delta(\epsilon,t)$, then
\begin{align*}
\max\{\|e^{A|t-s|}\phi-\phi\|,\|e^{B|t-s|}\phi-\phi\|,|t-s|\}<\frac{\epsilon}{M+M_{2}+\|B\phi\|M_{2}}.
\end{align*}
Let $\omega\in\Omega$ be given and assume $|t-s|<\delta(\epsilon,t)$. If $\omega$ is such that $\eta(s,t)(\omega)\ne 1$, then the result is immediate. Suppose $\eta(s,t)(\omega)=1$. If $\sigma$ denotes the switching time between $s$ and $t$, then
\begin{align*}
\left|\<\phi,u(t,\omega)-u(s,\omega)\>\right|
& \le \big|\<\phi,u(t,\omega)-u(\sigma,\omega)\>\big|
     +\big|\<\phi,u(\sigma,\omega)-u(s,\omega)\>\big|< 3\epsilon.\qquad\endproof
\end{align*}

%%%%%%%%%%%%%%%%%%%%%%%%%%%%%%%%%%%%%%%%%%%%%%%%%%%%%%%%%%%%

{\em Proof of Theorem \ref{exppde}}.
We seek to differentiate $\E \<\phi,u(t)\>$ with respect to $t$. Define $$f(t,\omega) = \<\phi,u(t,\omega)\>.$$
Let $h_{n}\to 0$ as $n\to\infty$. For a given $t_{0}>0$, define the difference quotient
\begin{align*}
g_{n}(\omega) & := \frac{1}{h_{n}}\left(f(t_{0}+h_{n},\omega)-f(t_{0},\omega)\right)\\
& = g_{n}(\omega)1_{\eta(t_{0}+h_{n},t_{0})=0} + g_{n}(\omega)1_{\eta(t_{0}+h_{n},t_{0})=1} + g_{n}(\omega)1_{\eta(t_{0}+h_{n},t_{0})\ge2}\\
 & = \Psi_{0}+\Psi_{1}+\Psi_{2},
\end{align*}
where $\eta$ is defined in Equation~(\ref{eta definition}). We will handle each of these terms differently.

We first consider $\Psi_0$. Assume $\omega$ is such that $t_{0}$ is not a switching time. 
By Lemma~\ref{differentiate},
\begin{align*}
\frac{1}{h_{n}}\left(f(t_{0}+h_{n},\omega)-f(t_{0},\omega)\right) \to \frac{d}{dt}f(t_{0},\omega)=\<A\phi,u(t_{0})\>\text{ as }n\to\infty.
\end{align*} 
Also observe that for such an $\omega$, we have that $1_{\eta(t_{0}+h_{n},t_{0})=0}(\omega)=1$ for $n$ sufficiently large. Since $\mu_{0}$ and $\mu_{1}$ are continuous distributions, this set of $\omega$'s has probability 1, and thus 
\begin{align*}
\Psi_0 = \frac{1}{h_{n}}\left(f(t_{0}+h_{n},\omega)-f(t_{0},\omega)\right)1_{\eta(t_{0}+h_{n},t_{0})=0} \to\<A\phi,u(t_{0})\>\quad\text{a.s. as }n\to\infty.
\end{align*}
We now apply the bounded convergence theorem to $\Psi_0$. 
 Let $n$ and $\omega$ be given. If $\eta(t_{0}+h_{n},t_{0})(\omega)\ne0$, then $|\Psi_0|= 0$, trivially. If $\eta(t_{0}+h_{n},t_{0})(\omega)=0$, then $f(t,\omega)$ is differentiable in $t$ for all $t\in(t_{0},t_{0}+h_{n})$. Therefore we can employ the mean value theorem to obtain 
\begin{align*}
\left|\frac{1}{h_{n}}\left(f(t_{0}+h_{n},\omega)-f(t_{0},\omega)\right)\right|
 &\le \sup_{t\in(t_{0},t_{0}+h_{n})}\left|\frac{d}{dt}f(t,\omega)\right|&\\
&= \sup_{t\in(t_{0},t_{0}+h_{n})}|\<A\phi,u(t,\omega)\>|\le \|A\phi\|M,
\end{align*}
since $\|u(t)\|\le M$ by assumption. Thus $|\Psi_0|\le \|A\phi\|M$ almost surely and so by the bounded convergence theorem, $\E \Psi_0\to \E \<A\phi,u(t_{0})\>$ as $n\to\infty$.

To complete the proof, we need only show that $\Psi_1$ and $\Psi_2$ both tend to 0 in mean as $n\to\infty$. We first work on $\Psi_1$. Observe that
\begin{align*}
\E |\Psi_1|
& = \E \left|\frac{1}{h_{n}}\left(f(t_{0}+h_{n},\omega)-f(t_{0},\omega)\right)1_{\eta(t_{0}+h_{n},t_{0})=1}\right|\\
& \le \frac{1}{h_{n}}\text{ess}\,\text{sup}_{\omega}\left|(f(t_{0}+h_{n},\omega)-f(t_{0},\omega))1_{\eta(t_{0}+h_{n},t_{0})=1}\right|\E \left(1_{\eta(t_{0}+h_{n},t_{0})=1}\right).
\end{align*}
It follows from Lemma~\ref{unif cts2} that $\text{ess}\,\text{sup}_{\omega}\left|(f(t_{0}+h_{n},\omega)-f(t_{0},\omega))1_{\eta(t_{0}+h_{n},t_{0})=1}\right|\to0$ as $n\to\infty$. Since by assumption $\P\left(\eta(t_{0}+h_{n},t_{0})=1\right)=O(h_{n})$, we conclude that $\E |\Psi_1|\to0$ as $n\to\infty$.

Finally, we consider $\Psi_2$. By the assumption that $\|u(t)\|\le M$,
\begin{align*}
\E |\Psi_2|
& = \E \left|\frac{1}{h_{n}}\left(f(t_{0}+h_{n},\omega)-f(t_{0},\omega)\right)1_{\eta(t_{0}+h_{n},t_{0})\ge2}\right|\\
& \le \frac{2\|\phi\|M}{h_{n}}\P\left(\eta(t_{0}+h_{n},t_{0})\ge2\right).
\end{align*}
By assumption, $\P\left(\eta(t_{0}+h_{n},t_{0})\ge2\right)=o(h_{n})$, and hence $\E |\Psi_2|\to0$ as $n\to\infty$.

Therefore
\begin{align*}
\frac{\E \<\phi,u(t_{0}+h_{n})\>-\E \<\phi,u(t_{0})\>}{h_{n}}
 = \E g_{n}
 \to \E \<A\phi,u(t_{0})\>\text{ as }n\to\infty.
\end{align*}
Since $h_{n}$ was an arbitrary sequence tending to 0 and $t_{0}$ was an arbitrary positive number, we conclude that $\frac{d}{dt}\E \<\phi,u(t)\> = \E \<A\phi,u(t)\>$ for all $t>0$.

Since taking the inner product against $\phi$ or $A \phi$ are both bounded linear operators on $H$, we can exchange expectation with inner product to obtain
\begin{align*}
\frac{d}{dt}\<\phi,\E u(t)\>  = \frac{d}{dt}\E \<\phi,u(t)\> = \E \<A \phi,u(t)\> = \<A \phi,\E u(t)\>.\qquad\endproof
\end{align*}

We now show that the mean at large time satisfies the homogeneous PDE.

\begin{theorem}\label{thm: mean limit}
Let $\phi\in D(A)\cap D(B)$. Then $\E u(t)\to \E \bar{u}$ weakly in $H$ as $t\to\infty$, where $\bar{u}$ 
is as in Theorem~\ref{conv in dist}. Furthermore, $\E \bar{u}$ satisfies
\begin{align*}
%(A\phi,h)=0.
\<A\phi,\E \bar{u}\>=0.
\end{align*}
\end{theorem}
\begin{remark}
{\rm
We will often assume the differential operator and the domain to be sufficiently regular so that this theorem implies $\E \bar{u}$ is a $C^{\infty}$ function satisfying the PDE pointwise.}
\end{remark}

\begin{proof}
Since the switching time distributions, $\mu_{0}$ and $\mu_{1}$, are assumed to be continuous distributions, they are non-arithmetic. Hence, by Theorem~\ref{conv in dist}, $\E g(u(t))\to \E g(\bar{u})$ as $t\to\infty$ for every continuous and bounded $g:H\to\mathbb{R}$.
For any $\eta\in H$, the function $\<\eta,\cdot\>:H\to\mathbb{R}$ is continuous and since by assumption, $\|u(t)\|\le M$ a.s., it follows that
\begin{align}\label{eqn:weak conv}
\E \<\eta,u(t)\>\to \E \<\eta,\bar{u}\>\quad\text{as }t\to\infty.
\end{align}
Since taking the inner product against $\eta$ is a bounded linear operator on $H$, we can exchange expectation with inner product in Equation~(\ref{eqn:weak conv}) above.
Hence, $\E u(t)\to \E \bar{u}$ weakly in $H$ as $t\to\infty$.

Of course it follows that in particular
\begin{align*}
\<\phi,\E u(t)\> \to \<\phi,\E \bar{u}\>\quad\text{and}\quad\<A\phi,\E u(t)\> \to \<A\phi,\E \bar{u}\>\quad\text{as }t\to\infty.
\end{align*}
By Theorem \ref{exppde}, $\frac{d}{dt}\<\phi,\E u(t)\> = \<A\phi,\E u(t)\>$. Thus, $\<\phi,\E u(t)\>$ and $\frac{d}{dt}\<\phi,\E u(t)\>$ both converge as $t\to\infty$ and so we conclude that $\frac{d}{dt}\<\phi,\E u(t)\>$ must actually converge to 0. Hence, $\<A\phi,\E \bar{u}\>=0$.
\qquad\end{proof}

%%%%%%%%%%%%%%%%%%%%%%%%%%%%%%%%%%%%%%%%%%%%%%%%%%%%%%%%%%%%%%%

\section{Examples}\label{section: examples}

In this section we apply our results from Sections~\ref{hilbertchapter} and \ref{pdechapter} to the heat equation on the interval $[0,L]$. We impose an absorbing Dirichlet boundary condition at $x=0$ and a stochastically switching boundary condition at $x=L$. In Example~\ref{ex: dir neu}, we consider switching between a Dirichlet and a Neumann boundary condition at $x=L$. In Example~\ref{ex: dir dir}, we consider switching between two Dirichlet boundary conditions at $x=L$.
%\begin{revision}
As in Section~\ref{pdechapter}, We use $\E$ to denote the Bochner integral of $L^2[0,L]$-valued random variables and not the pointwise expectation of random functions.
%\end{revision}

\subsection{Example~\ref{ex: dir neu}: Dirichlet/Neumann switching}\label{dir neu}
Consider the stochastic process that solves
\begin{align}
\partial_{t}u & = D \Delta u \quad\text{in }(0,L)
\end{align}
and at exponentially distributed times switches between the boundary conditions
\begin{equation*}
\begin{aligned}[c]
\begin{cases}
u(0,t)  = 0\\
u_{x}(L,t)  = 0
\end{cases}
\end{aligned}
\qquad\text{and}\qquad
\begin{aligned}[c]
\begin{cases}
u(0,t)  = 0\\
u(L,t)  = b>0.
\end{cases}
\end{aligned}
\end{equation*}

To cast this problem in the setting of previous sections, we set our Hilbert space to be $L^{2}[0,L]$ and define the operators
\begin{align*}
Au & := \Delta u\quad\text{if }u\in D(A)  := \left\{\phi\in H^{2}(0,L):\frac{\partial\phi}{\partial \textbf{n}}(L)=0=\phi(0)\right\}\\
Bu & := \Delta u\quad\text{if }u\in D(B)  := H_{0}^{1}(0,L)\cap H^{2}(0,L).
\end{align*}
We set $c=\frac{b}{L}x\in L^{2}[0,L]$ and let our switching time distributions, $\mu_{0}$ and $\mu_{1}$, be exponential with respective rate parameters $r_{0}$ and $r_{1}$. Let $u(t,\omega)$ be the $H$-valued process defined in Equation~(\ref{process definition}) with 
\begin{align}
\Phi^{1}_{t}(f) = e^{At}f
\qquad\text{and}\qquad
\Phi^{0}_{t}(f) = e^{Bt}(f-c)+c.
\end{align}

We are interested in studying the large time distribution of $u(t)$. By Corollary~\ref{cor conv}, we have that $u(t)$ converges in distribution as $t\to\infty$ to the $L^{2}[0,L]$-valued random variable $\bar{u}$ defined in the statement of the corollary. By the definitions of $Y_{0}$ and $Y_{1}$ in Equation~(\ref{pullback 0}), it is immediate that $\bar{u}$ is almost surely smooth, and using Proposition~\ref{prop:invariant set}, it follows that $\bar{u}(x)\le\frac{b}{L}x$ almost surely for each $x\in[0,L]$. In this section, we will find the expectation of $\bar{u}$.
\begin{proposition}\label{prop:slope} The function $\E \bar{u}$ is affine with slope
\begin{align}\label{prop:dir/neu slope}
\left(1+\frac{\rho}{\gamma}\tanh(\gamma)\right)^{-1}\frac{b}{L}
\end{align}
where $\gamma=L\sqrt{(r_{0}+r_{1})/D }$ and $\rho=r_{0}/r_{1}$.
\end{proposition}

%%%%%%%%%%%%%%%%%%%%%%%%%%%%%%%%%%%%%%%%%%%%%%%%%%%%%%%%%%%%%%%%%%%%%%%%%%%%%%%%%

To prove this proposition, we will use the results from both Sections~\ref{hilbertchapter} and \ref{pdechapter}. It is immediate that all of the assumptions in Section~\ref{pde setup} are satisfied, except for one; we need to check that there exists a deterministic $M$ so that $\|u(t)\|\le M$ almost surely for all $t\ge0$. We show that and more in the following lemma.
\begin{lemma}\label{infinity bound2}
Under the assumptions of the current section, we have that
\begin{align*}
\|u(t)\|\le L\Big(\max\{\|u_{0}\|_{\infty},b\}\Big)^{2},
\end{align*}
where $\|\cdot\|_{\infty}$ denotes the $L^{\infty}[0,L]$ norm.
Furthermore,
\begin{equation*}
\begin{aligned}[c]
\|Y_{1}\|_{\infty}\le b
\end{aligned}
\quad\text{and}\quad
\begin{aligned}[c]
\|Y_{0}\|_{\infty}\le b\quad\text{almost surely}.
\end{aligned}
\end{equation*}
\end{lemma}
\begin{proof}
First note that $\|c\|_{\infty}=\|\frac{b}{L}x\|_{\infty}=b$. If $f\in L^{2}[0,L]$, then by the maximum principle, we have that for any $t\ge0$
\begin{align}\label{eq:bound}
\|e^{At}f\|_{\infty}\le \|f\|_{\infty}\quad\text{and}\quad\|e^{Bt}(f-c)+c\|_{\infty}\le \max\{b,\|f\|_{\infty}\}.
\end{align}
Hence, $\max\{\|u(t)\|_{\infty},b\}$ is non-increasing in $t$ and so the bound on $\|u(t)\|$ is proven.

Since $S:=\{f\in L^{2}[0,L]:\|f\|_{\infty}\le b\}$ is a closed set in $L^{2}[0,L]$, Equation~(\ref{eq:bound}) and Proposition~\ref{prop:invariant set} give the desired bounds on $\|Y_{1}\|_{\infty}$ and $\|Y_{0}\|_{\infty}$.
\qquad\end{proof}

As in Corollary~\ref{cor conv}, let $\bar{u}$ have the limiting distribution of $u(t)$ as $t\to\infty$. Then by Theorem \ref{thm: mean limit}, we have that $\E \bar{u}\in L^{2}[0,L]$ satisfies
$\<\Delta \phi,\E \bar{u}\>=0$ 
for each $\phi\in C_{0}^{\infty}(0,L)$. By the regularity of $\Delta$ on $[0,L]$,
%\begin{revision}
it follows that $\E \bar{u}$ is not just a weak solution but that it is actually a smooth classical solution and hence it is the affine function
%\end{revision}
\begin{align*}
(\E \bar{u})(x) & = sx+d
\end{align*}
for some $s,d\in\mathbb{R}$. By Corollary~\ref{cor conv} of Section~\ref{hilbertchapter}, we have that
\begin{align}\label{eq: h and section 12}
sx+d = p\E Y_{1}+(1-p)\E Y_{0}
\end{align}
where $p=r_{0}/(r_{0}+r_{1})$. We will use Equation~(\ref{eq: h and section 12}) to determine $s$ and $d$.
While both $Y_{0}$ and $Y_{1}$ are almost surely smooth functions, $\E Y_{0}$ and $\E Y_{1}$ are \emph{a priori} only elements of $L^{2}[0,L]$.
It can be shown that $\E Y_{0}$ and $\E Y_{1}$ are smooth functions, but we will instead take limits of test functions to avoid evaluating $\E Y_{0}$ and $\E Y_{1}$ at specific points in $[0,L]$.

Let $\{\phi_{n}\}_{n=1}^{\infty}$ be such that $\phi_{n}\in C^{\infty}_{0}(0,L)$ and $\|\phi_{n}\|_{L^{1}}=1$ for each $n$ and 
\begin{align*}
\lim_{n\to\infty}\<\phi_{n},f\>= f(0)
\end{align*} for each $f\in C[0,L]$. Since the inner product with $\phi_{n}$ is a bounded linear functional in $L^{2}[0,L]$, we can interchange expectation with inner product in Equation~(\ref{eq: h and section 12}) to obtain
\begin{align}
d
 = \lim_{n\to\infty}\left[\<\phi_{n},p\E Y_{1} + (1-p)\E Y_{0}\>\right]
 = \lim_{n\to\infty}\left[p\E \<\phi_{n},Y_{1}\> + (1-p)\E \<\phi_{n},Y_{0}\>\right].\label{use BCT2}
\end{align}
We want to exchange the limit with the expectations. 
To do this, first observe that $Y_{1}(x)$ and $Y_{0}(x)$ are each almost surely continuous functions of $x\in[0,L]$ with $Y_{0}(0)=0=Y_{1}(0)$ almost surely. Thus,
\begin{align*}
\lim_{n\to\infty}\<\phi_{n},Y_{0}\> = 0
\quad\text{and}\quad\lim_{n\to\infty}\<\phi_{n},Y_{1}\> = 0\quad\text{almost surely}.
\end{align*}
Using Lemma~\ref{infinity bound2} and the assumption that $\|\phi_{n}\|_{L^{1}}=1$ for each $n$, we have that 
 \begin{align*}
 |\<\phi_{n},Y_{0}\>|\le b
 \quad\text{and}\quad
 |\<\phi_{n},Y_{1}\>|\le b\quad\text{almost surely}.
 \end{align*}
So we apply the bounded convergence theorem to Equation~(\ref{use BCT2}) to obtain
\begin{align}\label{first argument}
d
 = p\E \lim_{n\to\infty}\<\phi_{n},Y_{1}\> + (1-p)\E \lim_{n\to\infty}\<\phi_{n},Y_{0}\>
   =  0.
\end{align}

We now find the slope $s$ of $\E \bar{u}$. Denote the orthonormal eigenbasis of $A$ by $\{a_{k}\}_{k=1}^{\infty}$ and corresponding eigenvalues by $\{-\alpha_{k}\}_{k=1}^{\infty}$. Since
$\sum_{k=1}^{n}\<a_{k},\E Y_{1}\>a_{k}$
converges to $\E Y_{1}$ in $L^{2}[0,L]$ as $n\to\infty$, we have that for any $\phi\in C^{\infty}_{0}(0,L)$
\begin{align}
\<\phi,sx\>  = \<\phi,p\E Y_{1}\> + (1-p)\<\phi,\E Y_{0}\>
 = p\Big\<\phi,\sum_{k=1}^{\infty}\<a_{k},\E Y_{1}\>a_{k}\Big\> + (1-p)\<\phi,\E Y_{0}\>.\label{series}
\end{align}

We will need the following lemma which is an immediate corollary of Proposition~\ref{invariance}.

\begin{lemma}\label{cor inv}
Under the assumptions of Section~\ref{dir neu}, we have that for each $k\in\mathbb{N}$
\begin{align*}
\E [e^{-\alpha_{k}\tau_{1}}]\<a_{k},\E Y_{0}\> = \<a_{k},\E Y_{1}\>.
\end{align*}
\end{lemma}

Combining this lemma with $sx=p\E Y_{1}+(1-p)\E Y_{0}$ and rearranging terms yields
\begin{align*}
\<a_{k},\E Y_{1}\> & = \E [e^{-\alpha_{k}\tau_{1}}]\frac{s\<a_{k},x\>}{p\E [e^{-\alpha_{k}\tau_{1}}]+(1-p)}.
\end{align*}
Plugging this into Equation~(\ref{series}) gives
\begin{align*}
\<\phi,sx\> & = p\Big\<\phi\;,\;\sum_{k=1}^{\infty}\E [e^{-\alpha_{k}\tau_{1}}]\frac{s\<a_{k},x\>}{p\E [e^{-\alpha_{k}\tau_{1}}]+(1-p)}a_{k}\Big\> + (1-p)\<\phi,\E Y_{0}\>
\end{align*}
Solving for $s$, we find that
\begin{align}\label{equation: s before limit}
s & = (1-p)\<\phi,\E Y_{0}\>\Big(\<\phi,x\>-p\Big\<\phi\;,\;\sum_{k=1}^{\infty}\E [e^{-\alpha_{k}\tau_{1}}]\frac{\<a_{k},x\>}{p\E [e^{-\alpha_{k}\tau_{1}}]+(1-p)}a_{k}\Big\>\Big)^{-1}
\end{align}

Let $\{\phi_{n}\}_{n=1}^{\infty}\in C^{\infty}_{0}(0,L)$ be such that $\|\phi_{n}\|_{L^{1}}=1$ for each $n$ and 
$\lim_{n\to\infty}\<\phi_{n},f\>= f(L)$
 for each $f\in C[0,L]$.
 %\begin{revision}
We claim that
\begin{align}\label{second argument}
\lim_{n\to\infty}\<\phi_{n},\E Y_{0}\> & = b.
\end{align}
To see this, first note that $Y_0$ is almost surely smooth and $Y_0(L)=b$ almost surely, so $\lim_{n\to\infty}\<\phi_n,Y_0\>=b$ almost surely. Further, the inner product with $\phi_n$ is a bounded linear functional so $\<\phi_n,\E Y_0\>=\E\<\phi_n,Y_0\>$. Finally, $\|\phi_n\|_{L^1}=1$ and $\|Y_{0}\|_{\infty}\le b$ almost surely by Lemma~\ref{infinity bound2} so the bounded convergence theorem gives Equation~(\ref{second argument}).
%\end{revision}
Now, we want to show that
\begin{equation}\label{eq: interchange}
\begin{aligned}
\lim_{n\to\infty}&\Big\<\phi_{n}\;,\;\sum_{k=1}^{\infty}\frac{\E [e^{-\alpha_{k}\tau_{1}}]\<a_{k},x\>}{p\E [e^{-\alpha_{k}\tau_{1}}]+(1-p)}a_{k}\Big\>
 = \sum_{k=1}^{\infty}\frac{\E [e^{-\alpha_{k}\tau_{1}}]\<a_{k},x\>}{p\E [e^{-\alpha_{k}\tau_{1}}]+(1-p)}a_{k}(L).
\end{aligned}
\end{equation}

To do this, we need to show that $\sum_{k=1}^{\infty}\frac{\E [e^{-\alpha_{k}\tau_{1}}]\<a_{k},x\>}{p\E [e^{-\alpha_{k}\tau_{1}}]+(1-p)}a_{k}(x)$
converges uniformly in $x$.
Note that for each $k$
\begin{align*}
a_{k}(x)=\sqrt{\frac2L}\sin{\left(\frac{(2k-1)\pi x}{2L}\right)}
\qquad\text{and}\qquad
\alpha_{k}=\frac{D (2k-1)^{2}\pi^{2}}{4L^{2}}.
\end{align*}
Hence, 
$\E [e^{-\alpha_{k}\tau_{1}}] \le 1$ and
$p\E e^{-\alpha_{k}\tau_{1}}+(1-p) \ge 1-p$.
Furthermore, 
\begin{align*}
\|a_{k}\|_{\infty} \le \sqrt{\frac2L}
\qquad\text{and}\qquad
\<a_{k},x\> & %\int_{0}^{L} a_{k}(x)x\,dx
= \frac{4\sqrt{2}L^{3/2}}{\pi^{2}}\frac{(-1)^{k+1}}{(2k-1)^{2}}.
\end{align*}
So for any $N\in\mathbb{N}$
\begin{align*}
\left\|\sum_{k=N}^{\infty}\frac{\E [e^{-\alpha_{k}\tau_{1}}]\<a_{k},x\>}{p\E [e^{-\alpha_{k}\tau_{1}}]+(1-p)}a_{k}(x)\right\|_{\infty}
& \le \sum_{k=N}^{\infty}\frac{\left|\<a_{k},x\>\right|}{1-p}\left\|a_{k}(x)\right\|_{\infty}\\
& = \sum_{k=N}^{\infty}\frac{16L}{(1-p)\pi^{2}(2k-1)^{2}}\to0\quad\text{as }N\to\infty.
\end{align*}
Hence Equation~(\ref{eq: interchange}) is verified, and thus by Equation~(\ref{equation: s before limit}) we have that
\begin{align*}
s & = \frac{(1-p)b}{L-p\sum_{k=1}^{\infty}\E [e^{-\alpha_{k}\tau_{1}}]\frac{\<a_{k},x\>}{p\E e^{-\alpha_{k}\tau_{1}}+(1-p)}a_{k}(L)}.
\end{align*}
Using the assumptions on $\tau_{0}$, $\tau_{1}$, $\alpha_{k}$, and $a_{k}$, and using a series simplification formula found in Mathematica (\cite{mathematica}), this becomes
\begin{align*}
s & = \left(1+\frac{\rho}{\gamma}\tanh(\gamma)\right)^{-1}\frac{b}{L}
\end{align*}
where $\gamma=L\sqrt{(r_{0}+r_{1})/D }$ and $\rho=r_{0}/r_{1}$. This expectation is much different than the expectation we obtain when switching between boundary conditions of the same type in the next example below.

%%%%%%%%%%%%%%%%%%%%%%%%%%%%%%%%%%%%%%%%%%%%%%%%%%%%%%%%%%%%%%

\subsection{Example~\ref{ex: dir dir}: Dirichlet/Dirichlet switching}\label{dir dir}

Consider the stochastic process that solves
\begin{align}
\partial_{t}u & = D\Delta u \quad\text{in }(0,L)
\end{align}
and at exponentially distributed times switches between the boundary conditions
\begin{equation*}
\begin{aligned}[c]
\begin{cases}
u(0,t)  = 0\\
u(L,t)  = 0
\end{cases}
\end{aligned}
\qquad\text{and}\qquad
\begin{aligned}[c]
\begin{cases}
u(0,t)  = 0\\
u(L,t)  = b>0.
\end{cases}
\end{aligned}
\end{equation*}
 
To cast this problem in the setting of previous sections, we set our Hilbert space to be $L^{2}[0,L]$ and define the operator
\begin{align*}
Bu  := \Delta u\quad\text{if }u\in D(B)  := H_{0}^{1}(0,L)\cap H^{2}(0,L).
\end{align*}
We set $c=\frac{b}{L}x\in L^{2}[0,L]$. Let our switching time distributions, $\mu_{0}$ and $\mu_{1}$, be exponential with respective rate parameters $r_{0}$ and $r_{1}$. Let $u(t,\omega)$ be the $H$-valued process defined in Equation~(\ref{process definition}) 
with 
\begin{align}\label{eqn: dir dir maps}
\Phi^{1}_{t}(f) = e^{Bt}f
\qquad\text{and}\qquad
\Phi^{0}_{t}(f) = e^{Bt}(f-c)+c.
\end{align}

We are interested in studying the large time distribution of $u(t)$. As in Example~\ref{ex: dir neu}, we can use Corollary~\ref{cor conv} to obtain that $u(t)$ converges in distribution as $t\to\infty$ to some $L^{2}[0,L]$-valued random variable $\bar{u}$ defined in the statement of the corollary, and use Proposition~\ref{prop:invariant set} to obtain that $\bar{u}(x)\le\frac{b}{L}x$ almost surely for each $x\in[0,L]$. And as in Example~\ref{ex: dir neu}, we can use Theorem~\ref{thm: mean limit} to find the expectation of $\bar{u}$. However, since this problem switches between boundary conditions of the same type, we will be able to obtain much more information about $\bar{u}$.
%%%%%%%%%%%%%%%%%%%%%%%%%%%%%%%%%%%%%%%%%%%%%%%%%%%%%%%%%%%%%%%

Switching between two boundary conditions of the same type is significantly simpler than switching between boundary conditions of different types. This is because the two solution operators that we use when switching between boundary conditions of the same type both employ the same semigroup and thus the same orthonormal eigenbasis. Hence, we only need to consider the projections of the stochastic process in this one basis.
In this example, the orthonormal eigenbasis and corresponding eigenvalues for $B$ are for $k\in\mathbb{N}$
\begin{align}\label{eqn: evals}
b_{k} = \sqrt{\frac{2}{L}}\sin\left(\frac{k\pi}{L}x\right)
\qquad\text{and}\qquad
-\beta_{k} = -D(k\pi/L)^{2}.
\end{align}

Observe that for each $k$, the Fourier coefficient $u_{k}(t):=\< b_{k},u(t)\>\in\mathbb{R}$ is the solution to a one-dimensional ODE with a randomly switching right-hand side. Specifically, if $J_{t}$ is the jump process defined in Equation~(\ref{z defn}), then in between jumps of $J_{t}$ the process $u_{k}(t)$ satisfies
\begin{align}
\frac{d}{dt}u_{k} & = -J_{t}\beta_{k}u_{k}-(1-J_{t})\beta_{k}(u_{k}-c_{k}),\notag\\
\text{where}\quad c_{k} & := \<b_{k},c\> 
= \frac{(-1)^{k+1}b\sqrt{2L}}{k\pi}\label{eqn: c}.
\end{align}

We can use previous results on one-dimensional ODEs with randomly switching right-hand sides (see \cite{hurth_2010} or \cite{boxma_/off_2005}) to determine the marginal distributions of the Fourier coefficients of the stationary $\bar{u}$.
For each $k$, the marginal distributions of the Fourier coefficients of $Y_{0}$ and $Y_{1}$ are given by
\begin{align}\label{eqn:beta dist}
\frac{\<b_{k},Y_{0}\>}{c_{k}} \sim\text{Beta}\left(\frac{r_1}{\beta_{k}}+1,\frac{r_0}{\beta_k}\right)
\quad\text{and}\quad
\frac{\<b_{k},Y_{1}\>}{c_{k}} \sim\text{Beta}\left(\frac{r_1}{\beta_{k}},\frac{r_0}{\beta_k}+1\right).
\end{align}
Combining this with Corollary~\ref{cor conv} gives the marginal distributions of the Fourier coefficients of $\bar{u}$.

From Equation~(\ref{eqn:beta dist}) and Corollary~\ref{cor conv}, we obtain
\begin{align}\label{eqn: dir dir mean}
\E \bar{u} = (1-p)\frac{b}{L}x,
\end{align}
where $p = r_{0}/(r_{0}+r_{1})$. Thus, the expectation of the process at large time is merely the solution to the time homogeneous PDE with boundary conditions given by the average of the two boundary conditions that the process switches between. 

To further illustrate the usefulness of Equation~(\ref{eqn:beta dist}), we calculate the $L^{2}$-variance of $\bar{u}$. 
It follows from Equation~(\ref{eqn: dir dir mean}) that
\begin{align}\label{eqn: l2 variance}
\E \|\bar{u}-\E \bar{u}\|^{2}
 = \E \|\bar{u}\|^{2} - \frac{L}{3}b^{2}(1-p)^{2}.
\end{align}
Now by Corollary~\ref{cor conv}, we have that $\E \|\bar{u}\|^{2} = p\E \|Y_{1}\|^{2} + (1-p)\E \|Y_{0}\|^{2}$.
Combining this with Equation~(\ref{eqn:beta dist}) we obtain
\begin{align}\label{eqn: no plug}
\E \|\bar{u}\|^{2}
 = \sum_{k=1}^{\infty}\frac{r_{1}(r_{1}+\beta_{k})}{(r_{0}+r_{1})(r_{0}+r_{1}+\beta_{k})}c_{k}^{2}.
\end{align}
After plugging in our values for $\beta_{k}$, $b_{k}$ and $c_{k}$ in Equation~(\ref{eqn: no plug}), using a series simplification formula found in Mathematica (\cite{mathematica}), and combining with Equation~(\ref{eqn: l2 variance}), we obtain the $L^{2}$-variance
\begin{align*}
\E \|\bar{u}-\E \bar{u}\|^{2}
 = \frac{b^{2}Dr_{1}r_{0}(\gamma\coth(\gamma)-1)}{L(r_{0}+r_{1})^{3}},
\end{align*}
where $\gamma = L\sqrt{r_{0}+r_{1}/D}$.

While Equation~(\ref{eqn:beta dist}) is useful, knowing the marginal distributions of the individual Fourier coefficients of $Y_{0}$ or $Y_{1}$ is of course not enough to find their joint distributions, and the one-dimensional ODE methods used to obtain Equation~(\ref{eqn:beta dist}) do not give information about these joint distributions. We can, however, use our machinery developed in Section~\ref{hilbertchapter} to study these joint distributions. 

First, we can use Corollary~\ref{cor conv} and Proposition~\ref{invariance} to obtain joint statistics of the components of $\bar{u}$. To illustrate, we will calculate $\E \<Y_0,b_{n}\>\<Y_0,b_{m}\>$. Proposition~\ref{invariance} gives
\begin{align*}
\E \<Y_0,b_{n}\>\<Y_0,b_{m}\> = \E \<e^{B\tau_{0}}(e^{B\tau_{1}}Y_0-c)+c,b_{n}\>\<e^{B\tau_{0}}(e^{B\tau_{1}}Y_0-c)+c,b_{m}\>,
\end{align*}
where $\tau_{0}$ and $\tau_{1}$ are independent exponential random variables with rates $r_{0}$ and $r_{1}$. After recalling some basic facts about exponential random variables and making some algebraic manipulations, we obtain that $\E \<Y_0,b_{n}\>\<Y_0,b_{m}\>$ is equal to
\begin{align*}  
\frac{(\beta_{m} + \beta_{n} + 
     r_1) ((\beta_{m} + \beta_{n}) (\beta_{m} + r_1) (\beta_{n} + 
        r_1) + (2 \beta_{m} \beta_{n} + (\beta_{m} + \beta_{n}) r_1) r_0)}{(\beta_{m} + \beta_{n}) (\beta_{m} + r_1 + 
     r_0) (\beta_{n} + r_1 + r_0) (\beta_{m} + \beta_{n} + r_1 + r_0)}c_{m}c_{n}.
\end{align*}
From this, we can readily compute the covariance of $\<Y_0,b_{n}\>$ and $\<Y_0,b_{m}\>$. Other joint statistics of the Fourier coefficients of $Y_{0}$ and $Y_{1}$ (and hence $\bar{u}$ by Corollary~\ref{cor conv}) are found in analogous ways.

Next, we can use Proposition~\ref{prop:invariant set} to show that $\bar{u}$ almost surely has a very specific structure.
\begin{proposition}
Let $b_{k}$ be as in Equation~(\ref{eqn: evals}), $c_{k}$ as in Equation~(\ref{eqn: c}), $\bar{u}$ be as in Corollary~\ref{cor conv}, and $\bar{u}_{k}:=\<b_{k},\bar{u}\>$. Then for $k<n$ and for almost all $\omega\in\Omega$
\begin{align*}
\left(\frac{\bar{u}_{k}(\omega)}{c_{k}}\right)^{\left(n/k\right)^{2}}
\le\frac{\bar{u}_{n}(\omega)}{c_{n}}\le
1-\left(1-\frac{\bar{u}_{k}(\omega)}{c_{k}}\right)^{\left(n/k\right)^{2}}.
\end{align*}
\end{proposition}

\begin{proof}
For each $k,n\in\mathbb{N}$, let $R_{k,n}$ be the closed planar region
enclosed by the following two planar curves: 
\begin{align*}
\{P_{k,n}(e^{-Bt}c):t\ge0\}
\quad\text{and}\quad
\{P_{k,n}(c-e^{-Bt}c):t\ge0\}.
\end{align*}
Define $S_{k,n}\subset L^{2}[0,L]$ by
\begin{align*}
S_{k,n} = \{f\in L^{2}[0,L]: P_{k,n}(f)\in R_{k,n}\}.
\end{align*}
It is straightforward to check that $S_{k,n}$ is invariant under
$\Phi^{0}_{t}$ and $\Phi^{1}_{t}$ defined in Equation~(\ref{eqn: dir
  dir maps}) for each $k,n\in\mathbb{N}$. Hence, $\cap_{k,n}S_{k,n}$
is invariant under $\Phi^{0}_{t}$ and $\Phi^{1}_{t}$ and we have by
Proposition~\ref{prop:invariant set} that $Y_{0}$ and $Y_{1}$ (and
hence $\bar{u}$ by Corollary~\ref{cor conv}) are almost surely
contained in $\cap_{k,n}S_{k,n}$. 

For $k<n$, observe that $R_{k,n}$ can be written as
\begin{align*}
R_{k,n} = \Bigg\{(x,y)\in\mathbb{R}^{2}:0\le \frac{x}{c_{k}}\le
1\text{ and }\Big(\frac{x}{c_{k}}\Big)^{\left(n/k\right)^{2}}\le
\frac{y}{c_{n}}\le
1-\Big(1-\frac{x}{c_{k}}\Big)^{\left(n/k\right)^{2}}\Bigg\}. 
\end{align*}
The desired result follows.
\qquad\end{proof}

Furthermore, we have the following regularity result on
$\bar{u}$. Notice that it implies that as we move to finer and finer
spatial scales by taking $k \rightarrow \infty$, there is asymptotically
only one piece of randomness which determines the fine scale structure.
\begin{proposition}\label{BC bounds}
Let $r<1/2$, $b_{k}$ be as in Equation~(\ref{eqn: evals}), $c_{k}$ as in Equation~(\ref{eqn: c}), $Y_{0}^{k}:=\<b_{k},Y_{0}\>$, and $Y_{1}^{k}:=\<b_{k},Y_{1}\>$. Then for each $\omega\in\Omega$, there exists an $M(\omega)$ so that
\begin{align*}
1-\frac{M(\omega)}{k^{r}}
\le & \frac{Y^{k}_{0}(\omega)}{c_{k}}\le
1+\frac{M(\omega)}{k^{r}}
\quad\text{and}\quad
-\frac{M(\omega)}{k^{r}}
\le  \frac{Y^{k}_{1}(\omega)}{c_{k}}\le
\frac{M(\omega)}{k^{r}}.
\end{align*}
\end{proposition}

\begin{proof}
For each $k$, define
\begin{align*}
A_{k}:=\Big\{\omega\in\Omega:\Big|\frac{Y_{0}^{k}(\omega)}{c_{k}}-\E \frac{Y_{0}^{k}}{c_{k}}\Big|>\frac{1}{k^{r}}\Big\}.
\end{align*}
By Chebyshev's inequality and Equation~(\ref{eqn:beta dist}), we have that
\begin{align*}
\P(A_{k})\le \frac{\text{Var}(Y_{0}^{k})}{c_{k}^{2}}k^{2r} & =\frac{\beta_{k} r_{0} (\beta_{k} + r_{1})}{(\beta_{k} + r_{0} + r_{1})^2 (2 \beta_{k} + r_{0} + r_{1})}k^{2r}
\sim k^{2(r-1)}\text{ as }k\to\infty.
\end{align*}
Thus if $r<1/2$, then $\sum_{k=1}^{\infty}\P(A_{k})<\infty$ and we conclude by the Borel-Cantelli Lemma that $\P(A_{k}\text{ infinitely often})=0$. Hence, for almost all $\omega\in\Omega$, we can choose an $M(\omega)$ so that for all $k$,
\begin{align*}
\frac{r_{1}+\beta_{k}}{r_{0}+r_{1}+\beta_{k}}-\frac{M(\omega)}{k^{r}}
\le \frac{Y^{k}_{0}(\omega)}{c_{k}}\le
\frac{r_{1}+\beta_{k}}{r_{0}+r_{1}+\beta_{k}}+\frac{M(\omega)}{k^{r}}.
\end{align*}
A similar argument shows that for almost all $\omega\in\Omega$, we can choose an $M(\omega)$ so that for all $k$,
\begin{align*}
\frac{r_{1}}{r_{0}+r_{1}+\beta_{k}}-\frac{M(\omega)}{k^{r}}
\le & \frac{Y^{k}_{1}(\omega)}{c_{k}}\le
\frac{r_{1}}{r_{0}+r_{1}+\beta_{k}}+\frac{M(\omega)}{k^{r}}
\end{align*}
Since $\beta_{k}\sim k^{2}$ as $k\to\infty$, the desired results follows.
\qquad\end{proof}

We can iterate this proposition to obtain the following result which shows that $Y_{0}^{k}$ and $Y_{1}^{k}$ 
depend essentially on only one switching time for large $k$. Note that we could continue to iterate this proposition to obtain similar bounds. Recall the definition of each $\omega\in\Omega$ in Equation~(\ref{omega definition}).
\begin{corollary}
Let $r<1/2$, $b_{k}$ be as in Equation~(\ref{eqn: evals}), $c_{k}$ as in Equation~(\ref{eqn: c}), $Y_{0}^{k}:=\<b_{k},Y_{0}\>$, and $Y_{1}^{k}:=\<b_{k},Y_{1}\>$. Then for each $\omega\in\Omega$, there exists an $M_{0}(\omega)$ depending only on $\{(\tau_{0}^{k+1},\tau_{1}^{k})\}_{k\ge1}$ and an $M_{1}(\omega)$ depending only on $\{(\tau_{0}^{k},\tau_{1}^{k+1})\}_{k\ge1}$ such that
\begin{align*}
1 - e^{-\beta_{k}\tau_{0}^{1}}\Big(\frac{M_{0}(\omega)}{k^{r}}+1\Big)
\le&\frac{Y_{0}^{k}(\omega)}{c_{k}}\le
1 + e^{-\beta_{k}\tau_{0}^{1}}\Big(\frac{M_{0}(\omega)}{k^{r}}-1\Big)\\
e^{-\beta_{k}\tau_{1}^{1}}\Big(1-\frac{M_{1}(\omega)}{k^{r}}\Big)
\le&\frac{Y_{1}^{k}(\omega)}{c_{k}}\le
e^{-\beta_{k}\tau_{1}^{1}}\Big(1+\frac{M_{1}(\omega)}{k^{r}}\Big)
\end{align*}
\end{corollary}

\begin{proof}
Let $\omega$ be given. Define $\sigma:\omega\to\omega$ by
\begin{align*}
\sigma(\omega) = \left((\tau^{2}_{0},\tau^{1}_{1}),(\tau^{3}_{0},\tau^{2}_{1}),(\tau^{4}_{0},\tau^{3}_{1}),\dots\right).
\end{align*}
Then by the definition of $Y_{0}^{k}$ and $Y_{1}^{k}$, we have that
\begin{align*}
\frac{Y_{0}^{k}(\omega)}{c_{k}} = 1 + e^{-\beta_{k}\tau_{0}^{1}}\Big(\frac{Y_{1}^{k}(\sigma(\omega))}{c_{k}}-1\Big)
\end{align*}
By Proposition~\ref{BC bounds}, there exists an $M(\sigma(\omega))$ so that
\begin{align*}
-\frac{M(\sigma(\omega))}{k^{r}}
\le  \frac{Y^{k}_{1}(\sigma(\omega))}{c_{k}}\le
\frac{M(\sigma(\omega))}{k^{r}}.
\end{align*}
Thus,
\begin{align*}
1 - e^{-\beta_{k}\tau_{0}^{1}}\Big(\frac{M(\sigma(\omega))}{k^{r}}+1\Big)
\le \frac{Y_{0}^{k}(\omega)}{c_{k}} \le
1 + e^{-\beta_{k}\tau_{0}^{1}}\Big(\frac{M(\sigma(\omega))}{k^{r}}-1\Big).
\end{align*}
The bounds on $Y_{1}^{k}$ are proved in a similar way.
\qquad\end{proof}

%%%%%%%%%%%%%%%%%%%%%%%%%%%%%%%%%%%%%%%%%%%%%%%%%%%%%%%%%%%%%%%

\subsection{Application to insect physiology}\label{section:insect}

Essentially all insects breathe via a network of tubes that allows
oxygen and carbon dioxide to diffuse to and from their cells
\cite{wigglesworth_respiration_1931}. Air enters and exits this
network through valve-like holes (called spiracles) in the
exoskeleton. These spiracles regulate air flow by opening and
closing. Surprisingly, spiracles have three distinct phases of
activity, each typically lasting for hours. There is a completely
closed phase, a completely open phase, and a flutter phase in which
the spiracles rapidly open and close
\cite{lighton_discontinuous_1996}.

Insect physiologists have proposed at least five major hypotheses to
explain the purpose of this behavior
\cite{chown_discontinuous_2006}. In order to address these competing
hypotheses, physiologists would like to understand how much cellular
oxygen uptake decreases as a result of the spiracles' closing.

To answer this question, we consider the following model problem. We
represent a tube by the interval $[0,L]$ and model the oxygen
concentration at a point $x\in[0,L]$ at time $t$ by the function
$u(x,t)$. As diffusion is the primary mechanism for oxygen movement in
the tubes (see \cite{loudon_tracheal_1989}), the function $u$
satisfies the heat equation with some diffusion coefficient $D$. We
impose an absorbing boundary condition at the left endpoint of the
interval to represent cellular oxygen absorption where the tube meets
the insect tissue. The right endpoint represents the spiracle, and
since the spiracle opens and closes, the boundary condition here
switches between a no flux boundary condition, $u_{x}(L,t)=0$
(spiracle closed) and a Dirichlet boundary condition, $u(L,t)=b>0$
(spiracle open). We suppose that the spiracle switches from open to
closed and from closed to open with exponential rates $r_{0}$ and
$r_{1}$ respectively.

Then, the oxygen concentration $u(x,t)$ is the same process described above in Secion~\ref{dir neu}. Using the results from that section, if we let $\rho=r_{0}/r_{1}$ and $\gamma=L\sqrt{(r_{0}+r_{1})/D}$, then it follows from Proposition~\ref{prop:slope} that the oxygen flux to the cells at large time is given by
\begin{align*}
\left(1+\frac\rho\gamma\tanh(\gamma)\right)^{-1}\frac{bD}{L}.
\end{align*}
This formula is noteworthy because it shows that the cellular oxygen uptake not only depends on the average proportion of time the spiracle is open, but it also depends on the overall rate of opening and closing. In particular, note that if we keep the ratio $\rho$ fixed, but let $\gamma$ become large, then the oxygen uptake approaches $\frac{bD}{L}$. The biological meaning is that the insect can have its spiracles open an arbitrarily small proportion of time, and yet receive essentially just as much oxygen as if its spiracles were always open if it opens and closes with a sufficiently high frequency. This is important biologically, because it is almost certainly the correct explanation for fluttering.

\medskip

{\bf Acknowledgements.}  JCM would like to thank Yuri Bakhtin for
stimulating discussions.

\bibliography{paper2bib}

\begin{thebibliography}{10}

\bibitem{bakhtin_burgers_2007}
{\sc Y.~Bakhtin}, {\em Burgers equation with random boundary conditions}, Proc.
  Amer. Math. Soc., 135 (2007), p.~2257–2262.

\bibitem{bakhtin12}
{\sc Y.~Bakhtin and T.~Hurth}, {\em Invariant densities for dynamical systems
  with random switching}, Nonlinearity, 25 (2012).

\bibitem{balde_note_2009}
{\sc M.~Balde, U.~Boscain, and P.~Mason}, {\em A note on stability conditions
  for planar switched systems}, Internat. J. Control, 82 (2009),
  pp.~1882--1888.

\bibitem{belykh13}
{\sc I.~Belykh, V.~Belykh, R.~Jeter, and M.~Hasler}, {\em Multistable randomly
  switching oscillators: the odds of meeting a ghost}, Eur. Phys. J. Spec.
  Top.,  (2013).

\bibitem{benaim12qual}
{\sc M.~Benaim, S.~Leborgne, F.~Malrieu, and P.A. Zitt}, {\em Qualitative
  properties of certain piecewise deterministic markov processes}, preprint,
  (2012).

\bibitem{benaim12quant}
{\sc M.~Benaim, S.~Leborgne, F.~Malrieu, and P.~A. Zitt}, {\em Quantitative
  ergodicity for some switched dynamical systems}, Electron. Commun. Probab.,
  17 (2012), pp.~1--14.

\bibitem{billingsley_convergence_1999}
{\sc P.~Billingsley}, {\em Convergence of Probability Measures}, Wiley,
  Hoboken, 2nd~ed., 1999.

\bibitem{boxma_/off_2005}
{\sc O.~Boxma, H.~Kaspi, O.~Kella, and D.~Perry}, {\em On/off storage systems
  with state-dependent input, output, and switching rates}, Probab. Engrg.
  Inform. Sci, 19 (2005).

\bibitem{bressloff_metastability_2013}
{\sc P.~Bressloff and J.~Newby}, {\em Metastability in a stochastic neural
  network modeled as a velocity jump markov process}, SIAM J. Appl. Dyn. Syst.,
  12 (2013), pp.~1394--1435.

\bibitem{buckwar_exact_2011}
{\sc E.~Buckwar and M.~Riedler}, {\em An exact stochastic hybrid model of
  excitable membranes including spatio-temporal evolution}, J. Math. Biol., 63
  (2011), pp.~1051--1093.

\bibitem{chown_discontinuous_2006}
{\sc S.~Chown, A.~Gibbs, S.~Hetz, C.~Klok, John~R. Lighton, and E.~Marais},
  {\em Discontinuous gas exchange in insects: A clarification of hypotheses and
  approaches}, Physiol. Biochem. Zool., 79 (2006), pp.~333--343.

\bibitem{hairer13}
{\sc B.~Cloez and M.~Hairer}, {\em Exponential ergodicity for markov processes
  with random switching}, Bernoulli,  (to appear).

\bibitem{Crauel01}
{\sc H.~Crauel}, {\em Random point attractors versus random set attractors}, J.
  London Math. Soc. (2), 63 (2001), pp.~413--427.

\bibitem{CrauelFlandoli94}
{\sc H.~Crauel and F.~Flandoli}, {\em Attractors for random dynamical systems},
  Probab. Theory Related Fields, 100 (1994), pp.~365--393.

\bibitem{da_prato_1993}
{\sc G.~Da~Prato and J.~Zabczyk}, {\em Evolution equations with white-noise
  boundary conditions}, Stoch. and Stoch. Rep., 42 (1993), pp.~431--459.

\bibitem{diaconis_iterated_1999}
{\sc P.~Diaconis and D.~Freedman}, {\em Iterated random functions}, {SIAM}
  Rev., 41 (1999), p.~45–76.

\bibitem{doob1953}
{\sc J.~L. Doob}, {\em Stochastic processes}, Wiley Classics Library, John
  Wiley \& Sons, Inc., New York, 1990.
\newblock Reprint of the 1953 original, A Wiley-Interscience Publication.

\bibitem{duan_recent_2010}
{\sc J.~Duan, S.~Luo, and C.~Wang}, {\em Hyperbolic equations with random
  boundary conditions}, in Recent development in Stochastic Dynamics and
  Stochastic Analysis, vol.~8, World Scientific, 2010.

\bibitem{Duflo}
{\sc M.~Duflo}, {\em Random iterative models}, vol.~34 of Applications of
  Mathematics (New York), Springer-Verlag, Berlin, 1997.
\newblock Translated from the 1990 French original by Stephen S. Wilson and
  revised by the author.

\bibitem{durrett_probability_2010}
{\sc R.~Durrett}, {\em Probability : Theory and Examples}, Cambridge University
  Press, Cambridge, 4th~ed., 2010.

\bibitem{farkas_pathogen_2011}
{\sc J.~Farkas, P.~Hinow, and J.~Engelstadter}, {\em Pathogen evolution in
  switching environments: A hybrid dynamical system approach}, Math. Biosci.,
  240 (2012), pp.~70--75.

\bibitem{feldman}
{\sc R.~Feldman, K.~Meyer, and L.~Quenzer}, {\em Principles of
  Neuropharmacology}, Sinauer Assoc., Sunderland, Mass., 1997.

\bibitem{fuxe}
{\sc K.~Fuxe, A.~B. Dahlstrom, G.~Jonsson, D.~Marcellino, M.~Guescini, M.~Dam,
  P.~Manger, and L.~Agnati}, {\em The discovery of central monoamine neurons
  gave volume transmission to the wired brain}, Prog. Neurobiol., 90 (2010),
  pp.~82--100.

\bibitem{hasler13finite}
{\sc M.~Hasler, V.~Belykh, and I.~Belykh}, {\em Dynamics of stochastically
  blinking systems. part i: Finite time properties}, SIAM J. Appl. Dyn. Syst.,
  12 (2013), pp.~1007--1030.

\bibitem{hasler13asymptotic}
\leavevmode\vrule height 2pt depth -1.6pt width 23pt, {\em Dynamics of
  stochastically blinking systems. part ii: Asymptotic properties}, SIAM J.
  Appl. Dyn. Syst., 12 (2013), pp.~1031--1084.

\bibitem{hurth_2010}
{\sc T.~Hurth}, {\em Limit theorems for a one-dimensional system with random
  switchings}, master's thesis, Georgia Institute of Technology, December 2010.

\bibitem{Kifer86}
{\sc Y.~Kifer}, {\em Ergodic theory of random transformations}, vol.~10 of
  Progress in Probability and Statistics, Birkh\"auser Boston, Inc., Boston,
  MA, 1986.

\bibitem{Kifer88}
\leavevmode\vrule height 2pt depth -1.6pt width 23pt, {\em Random perturbations
  of dynamical systems}, vol.~16 of Progress in Probability and Statistics,
  Birkh\"auser Boston, Inc., Boston, MA, 1988.

\bibitem{lawleyode}
{\sc S.~D. Lawley, J.~C. Mattingly, and M.~C. Reed}, {\em Sensitivity to
  switching rates in stochastically switched {ODE}s}, Commun. Math. Sci., 12
  (2014), pp.~1343--1352.

\bibitem{lighton_discontinuous_1996}
{\sc J.~Lighton}, {\em Discontinuous gas exchange in insects}, Ann. Rev. of
  Ent., 41 (1996), pp.~309--324.
\newblock {PMID:} 8546448.

\bibitem{lin09}
{\sc H.~Lin and P.~J. Antsaklis}, {\em Stability and stabilizability of
  switched linear systems: A survey of recent results}, {IEEE} Trans. Automat.
  Contr., 54 (2009).

\bibitem{loudon_tracheal_1989}
{\sc C.~Loudon}, {\em Tracheal hypertrophy in mealworms: Design and plasticity
  in oxygen supply systems}, J. Exp. Bio., 147 (1989), pp.~217--235.

\bibitem{Mattingly99}
{\sc J.~C. Mattingly}, {\em Ergodicity of {$2$}{D} {N}avier-{S}tokes equations
  with random forcing and large viscosity}, Comm. Math. Phys., 206 (1999),
  pp.~273--288.

\bibitem{Mattingly02}
\leavevmode\vrule height 2pt depth -1.6pt width 23pt, {\em Contractivity and
  ergodicity of the random map {$x\mapsto\vert x-\theta\vert $}}, Teor.
  Veroyatnost. i Primenen., 47 (2002), pp.~388--397.

\bibitem{mattinglysuidan05}
{\sc J.~C. Mattingly and T.~M. Suidan}, {\em The small scales of the stochastic
  {N}avier-{S}tokes equations under rough forcing}, J. Stat. Phys., 118 (2005),
  pp.~343--364.

\bibitem{mattinglysuidan08}
\leavevmode\vrule height 2pt depth -1.6pt width 23pt, {\em Transition measures
  for the stochastic {B}urgers equation}, in Integrable systems and random
  matrices, vol.~458 of Contemp. Math., Amer. Math. Soc., Providence, RI, 2008,
  pp.~409--418.

\bibitem{newby_asymptotic_2011}
{\sc J.~Newby and J.~Keener}, {\em An asymptotic analysis of the spatially
  inhomogeneous velocity-jump process}, Multiscale Model. Simul., 9 (2011),
  pp.~735--765.

\bibitem{reed}
{\sc M.~Reed, H.F. Nijhout, and J.~Best}, {\em Projecting biochemistry over
  long distances}, Math. Mod. Natur. Phenom., 9 (2014), pp.~130--138.

\bibitem{ross_stochastic_1996}
{\sc S.~M. Ross}, {\em Stochastic processes}, Wiley, New York, 2nd~ed., 1996.

\bibitem{Schmalfuss96}
{\sc B.~Schmalfu{\ss}}, {\em A random fixed point theorem based on {L}yapunov
  exponents}, Random Comput. Dynam., 4 (1996), pp.~257--268.

\bibitem{random_pde}
{\sc R.~Schnaubelt and M.~Veraar}, {\em Stochastic equations with boundary
  noise}, in Parabolic Problems, J.~Escher, P.~Guidotti, M.~Hieber, P.~Mucha,
  J.~Pruss, Y.~Shibata, G.~Simonett, C.~Walker, and W.~Zajaczkowski, eds.,
  vol.~80 of Progr. Nonlinear Differential Equations Appl., Springer Basel,
  2011, pp.~609--629.

\bibitem{wang_reductions_2009}
{\sc W.~Wang and J.~Duan}, {\em Reductions and deviations for stochastic
  partial differential equations under fast dynamical boundary conditions},
  Stoch. Anal. Appl., 27 (2009), p.~431–459.

\bibitem{EKhaninMazelSinai}
{\sc E.~Weinan, K.~Khanin, A.~Mazel, and Y.~Sinai}, {\em Invariant measures for
  {B}urgers equation with stochastic forcing}, Ann. of Math. (2), 151 (2000),
  pp.~877--960.

\bibitem{wigglesworth_respiration_1931}
{\sc V.~B. Wigglesworth}, {\em The respiration of insects}, Bio. Rev., 6
  (1931), p.~181–220.

\bibitem{mathematica}
{\sc Inc. Wolfram~Research}, {\em Mathematica}, 2012.

\bibitem{Yin_2010}
{\sc G.~Yin and C.~Zhu}, {\em Hybrid {S}witching {D}iffusions}, Springer, New
  York, USA, 2010.

\bibitem{zhu_competitive_2009}
{\sc C.~Zhu and G.~Yin}, {\em On competitive {Lotka Volterra} model in random
  environments}, J. Math. Anal. Appl., 357 (2009), pp.~154--170.

\end{thebibliography}
\bibliographystyle{siam}

\end{document}